\documentclass[pdflatex,sn-mathphys-num]{sn-jnl}

\usepackage{graphicx}%
\usepackage{multirow}%
\usepackage{amsmath,amssymb,amsfonts}%
\usepackage{amsthm}%
\usepackage{mathrsfs}%
\usepackage[title]{appendix}%
\usepackage{xcolor}%
\usepackage{textcomp}%
\usepackage{manyfoot}%
\usepackage{booktabs}%
\usepackage{algorithm}%
\usepackage{algorithmicx}%
\usepackage{algpseudocode}%
\usepackage{listings}%
\usepackage{adjustbox}

\usepackage[utf8]{inputenc} 
\usepackage[T1]{fontenc}    
\usepackage{hyperref}       
\usepackage{url}            
\usepackage{booktabs}       
\usepackage{nicefrac}       
\usepackage{microtype}      
\usepackage{cleveref}       
\usepackage{doi}
\usepackage{anyfontsize}

\usepackage{tikz}
\usetikzlibrary{hobby}
\usepackage{pgfplots}
\pgfplotsset{compat=1.18}
\pgfplotsset{colormap name=viridis}
\usepgfplotslibrary{patchplots}
\pgfplotsset{
    legend image with text/.style={%
        legend image code/.code={%
            \node[anchor=center] at (0.3cm,0cm) {#1};
        }
    },
}

\usepackage{xcolor}
\allowdisplaybreaks     
\usepackage{mathtools}
\usepackage{subcaption}
\usepackage{cancel}
\usepackage{bm}
\usepackage{accents}
\usepackage{blindtext}

\usepackage{comment}

\newcommand{\bcom}[1]{\textcolor{olive}{\textbf{Bruno:} #1}}

\newcommand{\new}[1]{#1}

\newcommand{\R}{\mathbb{R}}

\newtheorem{theorem}{Theorem}
\newtheorem{proposition}{Proposition}%

\begin{document}

\title[A relax-fix-and-exclude algorithm for an MINLP problem with multilinear interpolations]{A relax-fix-and-exclude algorithm for an MINLP problem with multilinear interpolations}


\author*[1]{\fnm{Bruno M.} \sur{Pacheco}}\email{bruno.machado.pacheco@umontreal.ca}
\equalcont{The author was affiliated to the Department of Automation and Systems Engineering of the Federal University of Santa Catarina during the time this research was being developed.}

\author[2]{\fnm{Pedro M.} \sur{Antunes}}\email{antunespedromh@gmail.com}

\author[2]{\fnm{Eduardo} \sur{Camponogara}}\email{eduardo.camponogara@ufsc.br}

\author[2]{\fnm{Laio O.} \sur{Seman}}\email{laio.seman@ufsc.br}

\author[3]{\fnm{Vinícius R.} \sur{Rosa}}\email{viniciusrr@petrobras.com.br}

\author[3]{\fnm{Bruno F.} \sur{Vieira}}\email{bfv@petrobras.com.br}

\author[3]{\fnm{Cesar} \sur{Longhi}}\email{cesar.longhi@petrobras.com.br}

\affil[1]{\orgdiv{CIRRELT and Département d'Informatique et de Recherche Opérationnelle}, \orgname{Université de Montréal}, \orgaddress{\city{Montréal}, \state{QC}, \country{Canada}}}

\affil[2]{\orgdiv{Department of Automation and Systems Engineering}, \orgname{Federal University of Santa Catarina}, \orgaddress{\city{Florianópolis}, \state{SC}, \country{Brazil}}}

\affil[3]{\orgdiv{Petrobras Research Center}, \orgaddress{\city{Rio de Janeiro}, \state{RJ}, \country{Brazil}}}


\abstract{%
This paper introduces a novel algorithm for Mixed-Integer Nonlinear Programming (MINLP) problems with multilinear interpolations of look-up tables.
These problems arise when objectives or constraints contain black-box functions only known at a finite set of evaluations on a predefined grid.
We derive a piecewise-linear relaxation for the multilinear \new{interpolants, which require an MINLP formulation}.
Supported by the fact that our proposed relaxation \new{is as tight as possible}, we propose a novel algorithm that iteratively solves the MILP relaxation and refines the solution space through variable fixing and exclusion strategies.
This approach ensures convergence to an optimal solution, which we demonstrate, while maintaining computational efficiency.
We apply the proposed algorithm \new{(Relax-Fix-and-Exclude, RFE)} to a real-world offshore oil production optimization problem.
\new{Compared to the Gurobi solver, RFE was able to solve to optimality 20\% more instances within one hour.
Furthermore, on medium and large instances that both RFE and Gurobi solved to optimality, our algorithm took, on average, 52\% less time than Gurobi to converge, while on all instances that neither solved, RFE always achieved a lower primal-dual gap, on average 91\% smaller.}
}


\maketitle


\section{Introduction}\label{sec:intro}

\new{%
We consider mixed-integer programming problems where the cost and constraints involve black-box functions---functions known only at a finite number of evaluation points in their domain.
Such functions appear in many practical optimization tasks where the underlying models are complex simulations, result from empirical processes, or simply lack a tractable closed-form expression~\cite{jonesEfficientGlobalOptimization1998,bischlHyperparameterOptimizationFoundations2023,shieldsBayesianReactionOptimization2021,raoEngineeringOptimizationTheory2019a}.
We restrict our attention to the cases where the evaluations of the black-box functions are organized on a rectangular, axis-aligned grid, resulting in the so-called look-up tables~\citep{epelleComputationalPerformanceComparison2020a,grimstadGlobalOptimizationMultiphase2016}.
Consequently, to enable optimization over the continuous domain, an interpolation of the look-up table is required\footnote{To cover the entire domain properly, we would need to extrapolate the information on the look-up table. However, it is usual to assume that the optimal is within the boundaries of the table and thus interpolation suffices.} to generate a continuous function.
}

\new{%
A particularly popular interpolation method in this context is \emph{multilinear interpolation}~\citep{gastellu-etchegorryInterpolationProcedureGeneralizing2003,piniConsistentLookupTable2015,nellesGridBasedLookUpTable2000,nellesLinearPolynomialLookUp2020,bohnOptimizationbasedApproachCalibration2006a,guptaMonotonicCalibratedInterpolated2016,furlanSimpleApproachImprove2016,martinsImplementationMultilinearLookup2020}.
Unlike piecewise-linear approaches, the resulting interpolant preserves smoothness almost everywhere\footnote{The interpolant is continuous but not differentiable at the evaluation points, which are finitely many.}, does not require arbitrary subdivisions (triangulations) of the domain, and scales efficiently with dimensionality.
The downside is that embedding the interpolant in an optimization model introduces multilinear (i.e., non-convex) constraints, making the problem a challenging instance of mixed-integer nonlinear programming.
}

\new{%
This paper introduces a novel algorithm for mixed-integer nonlinear programming (MINLP) problems with functions modeled via multilinear interpolation.
We explore polyhedral tools for multilinear terms and show that they can be used to derive the convex hull of the nonlinearities arising from the interpolation method.
This convex-hull forms the basis of our \emph{Relax-Fix-and-Exclude} algorithm, exploiting it not only to compute tight bounds, but also to guide a depth-first search through variable fixing and generalized no-good cuts.
}

We demonstrate the performance of the proposed algorithm in an offshore oil production optimization problem.
The oil and gas industry is a classical domain application of MINLP over black-box functions, since optimization problems involve complex nonlinear functions representing processes such as reservoir pressure dynamics and multi-phase flow~\citep{beale_mathematical_1983}.
These functions are often derived from simulations or empirical data, making them difficult to incorporate directly into optimization problems.
Furthermore, the proposed approach is well-suited for these applications, where computational efficiency and accuracy are critical.

The key contribution of this work is the development of a piecewise-linear relaxation for the nonlinear interpolated functions. This relaxation leads to an efficient mixed-integer linear programming (MILP) approximation, which provides both a tractable formulation and a reliable lower bound for the original problem. By incorporating Special Ordered Sets of type 2 (SOS2), the method ensures that adjacent grid points are selected during interpolation, preserving the convexity and efficiency of the optimization process. Furthermore, the algorithm leverages fix-and-exclude strategies to refine the solution space progressively and converge to an optimal solution.

\subsection{Related Work}

Many real-world optimization problems involve functions with high evaluation costs, which are usually only known at finitely many points of their domain.
These problems have been known for a long time~\citep{jonesEfficientGlobalOptimization1998} and have appeared in a wide range of application areas, such as machine learning~\citep{bischlHyperparameterOptimizationFoundations2023}, chemistry~\citep{shieldsBayesianReactionOptimization2021}, and engineering~\citep{raoEngineeringOptimizationTheory2019a}.

A well-studied approach for solving such problems is to perform a piecewise-linear interpolation over the look-up tables of the black-box function~\citep{vielma2010modeling,Silva:EJOR:2014,huchette2022piecewise}.
In a grid-like partitioning of the domain, the piecewise-linear interpolation requires additional cuts to reduce the partitioning to simplices, such that the model becomes an injective mapping.
Although theoretically possible for any number of dimensions, this partitioning has only been proposed for domains of up to three dimension~\citep{misenerPiecewiseLinearApproximationsMultidimensional2010}.
On top of that, the number of simplices grows exponentially with the domain dimension, requiring exponentially many more integer variables.
Finally, the choice of simplices (within the hyperrectangular partition) is not an evident design choice, and it may have a significant impact on the approximation quality, e.g., by overestimating or underestimating the true function.

Multilinear interpolations of look-up tables are a common practice to deal with black-box functions or complex non-linear functions, with applications in physics simulation~\citep{gastellu-etchegorryInterpolationProcedureGeneralizing2003,piniConsistentLookupTable2015}, automotive software~\citep{nellesGridBasedLookUpTable2000,nellesLinearPolynomialLookUp2020,bohnOptimizationbasedApproachCalibration2006a}, and even machine learning~\citep{guptaMonotonicCalibratedInterpolated2016}.
For example, \citeauthor{furlanSimpleApproachImprove2016}~\cite{furlanSimpleApproachImprove2016} and \citeauthor{martinsImplementationMultilinearLookup2020}~\cite{martinsImplementationMultilinearLookup2020} replace phenomenological models of a chemical process with a look-up table that is interpolated during the simulation, which yielded improved efficiency at minor accuracy loss.
\new{%
In contrast to the piecewise-linear method, the multilinear interpolation is smooth everywhere on the domain except at a finite number of points.
Furthermore, it naturally scales to arbitrary domain dimensions and requires no design choice during implementation.
}

Optimizing with multilinearly interpolated look-up tables, in contrast to piecewise-linear interpolations, introduces multilinear terms, making the problem, at least, a mixed-integer quadratically constrained programming (MIQCP) problem. 
Optimization with multilinear constraints has been well-studied \cite{rikunConvexEnvelopeFormula1997,kehaModelsRepresentingPiecewise2004,McCormick1976} and continues to be an active area of research \cite{sundarPiecewisePolyhedralFormulations2021,kimPiecewisePolyhedralRelaxations2024}.
In fact, many solvers support multilinear constraints out-of-the-box \cite{gurobi,gleixnerExactFastAlgorithms2015}.
However, to the best of our knowledge, our work is the first to study the optimization problem with multilinear interpolations of look-up tables and to propose an exact algorithm for it.

\subsection{Contributions}

This work brings the following contributions to the technical literature:
\begin{itemize}

  \item \new{An MINLP approximation for a class of real-world problems involving black-box functions using multilinear interpolation with data from sensitivity analysis, which can be solved efficiently.}

  \item \new{A tight MILP relaxation of multilinear interpolation functions arising from piecewise-linear formulations.}

  \item A Relax-Fix-and-Exclude algorithm with convergence guarantee for solving optimization problems with multilinear interpolation functions, which iteratively solves piecewise-linear based relaxations (\textit{Relax}), optimizes over a reduced domain resulting from variable fixing (\textit{Fix}), and removes this domain from the search space (\textit{Exclude}).

  \item A computational demonstration of the efficiency of the relax-fix-and-exclude algorithm and comparison with \new{state-of-the-art MINLP solvers (Gurobi, SCIP, and BARON)} in an application to oil production in offshore platforms. 
\end{itemize}

\subsection{Organization}

Section \ref{problem} introduces the baseline problem, which consists of a mixed-integer nonlinear program with black-box functions approximated through multilinear interpolation of look-up tables, resulting in an MINLP formulation.
Section \ref{algorithm} presents the relax, fix, and exclude steps, which are then combined into the proposed algorithm.
Section \ref{sec:experiments} reports computational results from applying the relax-fix-and-exclude algorithm for oil production optimization in offshore platforms \new{(Section \ref{sec:oil-production-optimization})}, and a comparison with state-of-the-art \new{MINLP solvers Gurobi, SCIP and BARON}.
    Section \ref{sec:conclusion} draws final conclusions and suggests future research directions.


\section{Problem Statement} \label{problem}

This work is concerned with the following  mixed-integer program
\begin{equation}\label{eq:problem}
\begin{split}
    C = \min_{\bm{x},\bm{y}} \quad & \widetilde{f}_0(\bm{x}) + \bm{a}_0^{T} \bm{y}  \\
    \textrm{s.t.} \quad & \widetilde{f}_i(\bm{x}) + \bm{a}_i^{T} \bm{y} \le b_i,\,i=1,\ldots,m \\
      & \bm{x}\in X\subseteq\R^{n}, \bm{y}\in \{0,1\}^{p}
\end{split}
\end{equation}
where each
\begin{equation}
    \begin{split}
	\widetilde{f}_i: X &\longrightarrow \R \\
	\bm{x} &\longmapsto \widetilde{f}_i(\bm{x})
    \end{split}
\end{equation}
\new{is a multilinear interpolation of a black-box function $f_i:X \rightarrow \R$, for $i=0,\ldots,m$.}

\new{%
It is assumed that a \emph{look-up table} is available with evaluations of the black-box functions at a finite set of \emph{breakpoints} $\mathcal{B} \subseteq X$.
In other words, for each $\hat{\bm{x}} \in\mathcal{B}$, we know the value of $f_i(\hat{\bm{x}})$ for all $i=0,\ldots,m$.
}

\new{%
We assume the breakpoints form a rectangular grid $\mathcal{B}=\mathcal{B}_1\times \cdots\times \mathcal{B}_n$, such that each $\mathcal{B}_j = \left\{ \hat{x}_j^{0},\hat{x}_{j}^{1},\hat{x}_{j}^{2},\ldots \right\}$ is a finite set.
Let the breakpoints be indexed in increasing order $\hat{x}^{0}_j<\hat{x}^{1}_j<\hat{x}_{j}^{2}<\cdots$, for $j=1,\ldots,n$.
Further, let $\mathcal{I}_j$ denote the set of \emph{indices} of the breakpoints for the $j$-th variable, such that we can write $\mathcal{B}_j = \{\hat{x}_j^{k}\}_{k\in\mathcal{I}_j}$.
Then, we will let $\mathcal{I}=\mathcal{I}_1\times \cdots \times \mathcal{I}_n$ denote the set of \emph{multi-indices} of the breakpoints, such that $\forall \bm{k}=(k_1,\ldots,k_n)\in\mathcal{I}$, we can write $\hat{\bm{x}}^{\bm{k}} = (\hat{x}_1^{k_1},\ldots,\hat{x}_n^{k_n})\in \mathcal{B}$.
}

\subsection{Multilinear Interpolation}

To develop an understanding of such an interpolation method, let us consider a single black-box function $f: X \subseteq \R^{n} \rightarrow \R$.
\new{%
If $n=1$, we write $x_1=x$, and we have a traditional linear interpolation over $\mathcal{B}=\left\{ \hat{x}^{0},\hat{x}^{1},\hat{x}^{2},\ldots \right\} \subset \R$.
}
Without loss of generality, take $x\in \text{conv}( \mathcal{B} ) $ such that $\hat{x}^0\le x\le \hat{x}^1$.
Then, we can write the linear interpolation of $f(x)$ as
\begin{equation}
\begin{aligned}
\widetilde{f}(x) &= f(\hat{x}^0) + \frac{x - \hat{x}^0}{\hat{x}^1 - \hat{x}^0}(f(\hat{x}^1) - f(\hat{x}^0))  \\
	      &= \underbrace{\left(\frac{\hat{x}^1-x}{\hat{x}^1 - \hat{x}^0}\right)}_{\xi_{0}} f(\hat{x}^0) + \underbrace{\left(\frac{x - \hat{x}^0}{\hat{x}^1 - \hat{x}^0}\right)}_{\xi_{1}}f(\hat{x}^1)
.\end{aligned}
\end{equation}

\new{%
Note that the coefficients  $\xi_0$ and $\xi_1$ of the linear interpolation are such that $\xi_0+\xi_1=1$, and
\begin{equation}
x = \xi_{0}\hat{x}^0 + \xi_{1}\hat{x}^1
,\end{equation}
that is, they match the weights that express $x$ as a convex combination (CC) of two \emph{consecutive} breakpoints.
}
One way to formulate such a convex combination for the one-dimensional case is through the use of an SOS2 variable
\begin{equation}
\begin{aligned}
    x &= \sum_{k \in \mathcal{I}} \xi_{k}\hat{x}^{k} \\
    1 &= \sum_{k \in \mathcal{I}} \xi_{k} \\
      & \bm{\xi} \in \text{SOS2}(\mathcal{I}),\, \new{\bm{\xi} \ge \bm{0}},
\end{aligned}
\end{equation}
which ensures that only two consecutive breakpoints will have the respective coefficients assuming non-zero values.
\new{We use the notation $\text{SOS2}(\mathcal{I})$ to convey that the SOS2 constraint is applied over the ordering defined by $\mathcal{I}$.}
Then, we can write the linear interpolation as
\begin{equation}
\begin{aligned}
    \lambda_k &= \xi_{k},  & \forall k\in \mathcal{I}, \\
    \widetilde{f} &= \sum_{k\in \mathcal{I}} \lambda_k f(\hat{x}^{k})
.\end{aligned}
\end{equation}
\new{Note that we introduce the redundant $\lambda$ variables to differentiate $\xi_k$, the weight of the $k$-th breakpoint, from $\lambda_k$, the weight of the function evaluation at the $k$-th breakpoint, which happen to coincide in the onedimensional case.}
in which we introduce the $\lambda$ variables to differentiate the weights of the breakpoints from the weights of the evaluations of $f$.

\begin{figure}[h]
    \centering
    \begin{subfigure}[t]{0.49\textwidth}
    \centering
    \begin{tikzpicture}
        \begin{axis}[
            axis lines = center,
            point meta min=0.5,
            point meta max=2,
            view={15}{30},
            xmin=0.5,
            xmax=2.5,
            xtick={1,2},
            xticklabels={$\hat{x}^0$,$\hat{x}^1$},
            ymin=0.5,
            ymax=2.5,
            ytick={1,2},
            yticklabels={$\hat{y}^0$,$\hat{y}^1$},
            y tick label style={anchor=south},
            zlabel=$\widetilde{f}$,
            zmin=0,
            zmax=2,
            ztick={0},
            z label style={anchor=east},
            pin distance=0.3cm,
    	]
            \addplot3 [black,dashed] coordinates {
                (1,0.5,0) (1,2.5,0)
                
                (2,0.5,0) (2,2.5,0)
                
                (0.5,1,0) (2.5,1,0)
                
                (0.5,2,0) (2.5,2,0)
            };
            
    	    \addplot3[surf,domain=1:2,shader=interp] {8 - 4.5*x -4.5*y + 3*x*y};

            \addplot3 [black,dotted,mark=*,mark size=1pt] coordinates {
                (1,1,0) (1,1,2)
                
                (2,2,0) (2,2,2)
                
                (2,1,0) (2,1,0.5)
                
                (1,2,0) (1,2,0.5)
            };
            \node [coordinate,pin=right:{$f(\hat{x}^1,\hat{y}^1)$}] at (axis cs:2,2,2) {};
        \end{axis}
    \end{tikzpicture}
    \caption{Bilinear interpolation.}\label{fig:bilinear-interp}
    \end{subfigure}
    \begin{subfigure}[t]{0.49\textwidth}
    \centering
    \begin{tikzpicture}
        \begin{axis}[
            axis lines = center,
            point meta min=0.5,
            point meta max=2,
            view={15}{30},
            xmin=0.5,
            xmax=2.5,
            xtick={1,2},
            xticklabels={$\hat{x}^0$,$\hat{x}^1$},
            ymin=0.5,
            ymax=2.5,
            ytick={1,2},
            yticklabels={$\hat{y}^0$,$\hat{y}^1$},
            y tick label style={anchor=south},
            zlabel=$\widetilde{f}$,
            zmin=0,
            zmax=2,
            ztick={0},
            z label style={anchor=east},
	]
            \addplot3 [black,dashed] coordinates {
                (1,0.5,0) (1,2.5,0)
                
                (2,0.5,0) (2,2.5,0)
                
                (0.5,1,0) (2.5,1,0)
                
                (0.5,2,0) (2.5,2,0)
            };



            \addplot3[patch,patch refines=0,shader=faceted interp]  coordinates {
                (1,1,2)
                (1,2,0.5)
                (2,2,2)
                
                (1,2,0.5)
                (2,2,2)
                (2,1,0.5)

                (1,1,2)
                (1,2,0.5)
                (2,1,0.5)

                (1,1,2)
                (2,1,0.5)
                (2,2,2)
            };

            \addplot3 [black,dotted,mark=*,mark size=1pt] coordinates {(1,1,0) (1,1,2)};
            \addplot3 [black,dotted,mark=*,mark size=1pt] coordinates {(2,2,0) (2,2,2)};
            \addplot3 [black,dotted,mark=*,mark size=1pt] coordinates {(2,1,0) (2,1,0.5)};
            \addplot3 [black,dotted,mark=*,mark size=1pt] coordinates {(1,2,0) (1,2,0.5)};
        \end{axis}
    \end{tikzpicture}
    \caption{Piecewise linear relaxation.}\label{fig:pwl-relaxation}
    \end{subfigure}
    \caption{Example of multilinear interpolation and the respective piecewise-linear relaxation for a black-box function on a two-dimensional domain. The interpolation is illustrated on a single rectangle of the sensitivity analysis grid, namely $[\hat{x}^0,\hat{x}^1] \times [\hat{y}^0,\hat{y}^1]$.}
    \label{fig:bilinear-interp-example}
\end{figure}

For $n>1$, the multilinear interpolation over grids can be defined recursively~\citep{weiser_note_1988}.
A bidimensional example is illustrated in Figure~\ref{fig:bilinear-interp}.
We get into details for the case $n=3$.
To ease the notation, let $x_1=x$, $x_2=y$, and $x_3=z$.
Then, for $\left( x,y,z \right) \in \text{conv}(\mathcal{B}_x\times \mathcal{B}_y\times \mathcal{B}_z)$, assume, without loss of generality, that $\hat{x}^0,\hat{x}^1\in \mathcal{B}_x$, $\hat{y}^0,\hat{y}^1\in \mathcal{B}_y$, and $\hat{z}^0,\hat{z}^1\in \mathcal{B}_z$ are the vertices of the rectangular prism that contains $\left( x,y,z \right) $, i.e., $\hat{x}^0\le x\le \hat{x}^1$, $\hat{y}^0\le y\le \hat{y}^1$, and  $\hat{z}^0\le z\le \hat{z}^1$.
By interpolating each dimension recursively, the approximation can be written
\begin{equation}\label{eq:trilinear-recursive}
\begin{aligned}
    \widetilde{f} &= \frac{\hat{z}^1-z}{\hat{z}^1-\hat{z}^0}\widetilde{f}_{x,y}\!\left( \hat{z}^0 \right) + \frac{z-\hat{z}^0}{\hat{z}^1-\hat{z}^0}\widetilde{f}_{x,y}\!\left( \hat{z}^1 \right)   \\
    \widetilde{f}_{x,y}\!\left( \hat{z}^{0|1} \right)  &= \frac{\hat{y}^1-y}{\hat{y}^1-\hat{y}^0} \widetilde{f}_x\!\left( \hat{y}^0, \hat{z}^{0|1} \right) + \frac{y-\hat{y}^0}{\hat{y}^1-\hat{y}^0} \widetilde{f}_{x}\!\left( \hat{y}^1,\hat{z}^{0|1} \right)  \\
    \widetilde{f}_{x}\!\left( \hat{y}^{0|1},\hat{z}^{0|1} \right) &= \frac{\hat{x}^1-x}{\hat{x}^1-\hat{x}^0}f\!\left( \hat{x}^0,\hat{y}^{0|1},\hat{z}^{0|1} \right) + \frac{x-\hat{x}^0}{\hat{x}^1-\hat{x}^0} \widetilde{f}\!\left( \hat{x}^1,\hat{y}^{0|1},\hat{z}^{0|1} \right)
,\end{aligned}
\end{equation}
where $\hat{z}^{0|1}$ denotes either $\hat{z}^0$ or $\hat{z}^1$, and $\hat{y}^{0|1}$ likewise.
Expanding the approximation $\widetilde{f}$ above in terms of products, we can rewrite the interpolated function \(\widetilde{f}\) as a summation over products of the basis functions for each variable \( x \), \( y \), and \( z \), 
\begin{equation}
\new{\widetilde{f} = \sum_{(k,j,i)\in\{0,1\}^3} \xi^{x}_k \xi^{y}_j \xi^{z}_i  f(\hat{x}^k, \hat{y}^j, \hat{z}^i)} \label{eq:trilinear-recursive:prod}
\end{equation}
where
\begin{equation}
\begin{aligned}
\xi^x_{k} &= 
\begin{cases}
    \frac{\hat{x}^1 - x}{\hat{x}^1 - \hat{x}^0}, & \text{if } k = 0, \\
    \frac{x - \hat{x}^0}{\hat{x}^1 - \hat{x}^0}, & \text{if } k = 1,
\end{cases} \\
\xi^y_{j} &= 
\begin{cases}
    \frac{\hat{y}^1 - y}{\hat{y}^1 - \hat{y}^0}, & \text{if } j = 0, \\
    \frac{y - \hat{y}^0}{\hat{y}^1 - \hat{y}^0}, & \text{if } j = 1,
\end{cases}  \\
\xi^z_{i} &= 
\begin{cases}
    \frac{\hat{z}^1 - z}{\hat{z}^1 - \hat{z}^0}, & \text{if } i = 0, \\
    \frac{z - \hat{z}^0}{\hat{z}^1 - \hat{z}^0}, & \text{if } i = 1,
\end{cases}
\end{aligned}
\end{equation}
\new{and \(\xi^{x}_k\) serves as the basis function for variable $x$, with $k$ iterating over $0$ and $1$, and similarly for the other two variables}.

Like in the one-dimensional case, each $\xi$ is a coefficient of the CC of breakpoints that results in $\left( x,y,z \right) $, i.e.,
\begin{equation}
\begin{aligned}
    x &= \xi^x_{0}\hat{x}^0 + \xi^x_{1}\hat{x}^1 , \\
    y &= \xi^y_{0}\hat{y}^0 + \xi^y_{1}\hat{y}^1 , \\
    z &= \xi^z_{0}\hat{z}^0 + \xi^z_{1}\hat{z}^1 , \\
    1 &=  \xi^x_{0} +  \xi^x_{1} , \\
    1 &=  \xi^y_{0} +  \xi^y_{1} , \\
    1 &=  \xi^z_{0} +  \xi^z_{1} , \\
    0 &\leq \xi^x_{0}, \xi^x_{1}, \xi^y_{0}, \xi^y_{1}, \xi^z_{0}, \xi^z_{1}
.\end{aligned}
\end{equation}
Thus, \(\widetilde{f}\) is a weighted sum where each term is the product of basis functions for each variable, multiplied by the corresponding function value at that grid point, \(f(\hat{x}^k, \hat{y}^j, \hat{z}^i)\).
We can rewrite the interpolation \eqref{eq:trilinear-recursive} as
\begin{equation}
\widetilde{f} = \sum_{\left( k,j,i \right) \in \left\{ 0,1 \right\}^{3}} 
\lambda_{k,j,i}f(\hat{x}^k,\hat{y}^j,\hat{z}^i)
,
\end{equation}
where $\lambda_{k,j,i} = \xi^x_{k}\xi^y_{j}\xi^z_{i}$, for all $(k,j,i)\in\mathcal{I}$.

More generally, we can express the multilinear interpolation of a function of $n$ variables by using SOS2 vectors $\bm{\xi}^j$ as coefficients of the CC of the breakpoints $\mathcal{B}_j$ into $x_j$, for $j=1,\ldots,n$.
Precisely, we impose
\begin{equation}\label{eq:input-constraints}
\begin{aligned}
    x_{j} &= \sum_{k\in\mathcal{I}_j} \xi^j_{k} \hat{x}_{j}^{k} ,  \\
    1 &= \sum_{k\in\mathcal{I}_j} \xi^j_{k} , \\
      & \bm{\xi^j} \in \text{SOS2}(\mathcal{I}_j) 
,\end{aligned} 
\end{equation}
for each $j=1,\ldots,n$.
\new{Note that, in the above formulation, an SOS2 constraint is imposed at each vector variable $\bm{\xi}^j\in\R^{|\mathcal{I}_j|}$ individually.}
The interpolation weights are then
\begin{equation}\label{eq:interpolation-bilinears}
   \lambda_{\bm{k}} = \prod_{j=1}^{n} \xi^j_{k_j},\quad \forall\bm{k} \in \mathcal{I}
.\end{equation}
Finally, the formulation of $\widetilde{f}$ as a multilinear interpolation becomes
\begin{equation}\label{eq:mult-interp}
   \widetilde{f} = \sum_{\bm{k}\in\mathcal{I}} \lambda_{\bm{k}}  f(\hat{\bm{x}}^{\bm{k}})
.
\end{equation}

\subsection{MINLP \new{Formulation}}

By considering the formulation of the multilinear interpolant in \eqref{eq:mult-interp}, we can write problem \eqref{eq:problem} as an MINLP
\begin{equation}\label{eq:minlp}\tag{$\mathcal{P}$}
\begin{aligned}
    C = \min_{\bm{\xi},\bm{y}} \quad & \widetilde{f}_0 + \bm{c}^{T} \bm{y}  \\
    \textrm{s.t.} \quad & \widetilde{f}_i + \bm{a}_i^{T} \bm{y} \le b_i, & i=1,\ldots,m, \\
			& \widetilde{f}_i = \sum_{\bm{k}\in\mathcal{I}} \lambda_{\bm{k}}  f_i(\hat{\bm{x}}^{\bm{k}}), & i=0,\ldots,m, \\
   &\lambda_{\bm{k}} = \prod_{j=1}^{n} \xi^j_{k_j}, &
     \forall \bm{k}\in\mathcal{I}, \\
   &x_{j} = \sum_{k\in\mathcal{I}_j} \xi^j_{k} \hat{x}_{j}^{k}, &j=1,\ldots,n, \\
   &1 = \sum_{k\in\mathcal{I}_j} \xi^j_{k}, &j=1,\ldots,n, \\
      & \bm{\xi^j} \in \text{SOS2}(\mathcal{I}_j), \bm{\xi^j}\geq \mathbf{0} & j=1,\ldots,n, \\
      & \bm{x}\in X, \,\bm{y}\in \left\{ 0,1 \right\}^{p}
,\end{aligned}
\end{equation}
where each $f_i(\hat{\bm{x}}^{\bm{k}})$ is obtained from the look-up table.
\new{%
Variables $\bm{x}$, $\bm{\lambda}$, and $\widetilde{\bm{f}}$ are functionally dependent on $\bm{\xi}$ and are introduced solely for notational convenience.
In other words, \ref{eq:minlp} is effectively a problem over the $(\bm{\xi},\bm{y})$ extended space.
Furthermore, we have as many $\xi^j_k$ variables as breakpoints along each $j$-th dimension, that is, $\bm{\xi}^j\in \R^{|\mathcal{B}_j|}$, for $j=1,\ldots,n$.
}

\new{%
Note that the addition of constraints \eqref{eq:input-constraints} implies $\bm{x} \in \text{conv}(\mathcal{B}) \cap X$.
So it is an implicit assumption that the optimal solution to the problem can be expressed as a convex combination of the breakpoints.
}


\section{The Relax-Fix-and-Exclude Algorithm} \label{algorithm}

\new{Problem \ref{eq:minlp} is a nonconvex Mixed-Integer Quadratically-Constrained Quadratic Programming (MIQCQP) problem, which, despite being supported by modern solvers like Gurobi~\citep{gurobi}, BARON~\cite{sahinidisBARONGeneralPurpose1996,zhangSolvingContinuousDiscrete2024}, and SCIP~\cite{BolusaniEtal2024OO,BolusaniEtal2024ZR},} is a notoriously hard class of problems.
\new{%
Classical methods for this class of problems are based on sequentially solving MILP relaxations and NLP subproblems~\citep{BIEGLER20041169}.
Our solution approach for \ref{eq:minlp} is based on a tight MILP relaxation, which we derive in Section~\ref{sec:milp-relaxation}.
However, beyond just using the relaxation to compute tighter bounds, or refining it with cutting planes in an outer approximation fashion, we opt for an enumeration approach, using the relaxation to guide our exploration of the solution space, similar to a diving heuristic.
}

\new{%
The intuition behind our approach is that for smooth functions and a sufficiently dense set of breakpoints, the solution to the MILP relaxation will lie on the same hyperrectangle of the domain as the optimal solution to the original MINLP.
Thus, we use a solution of the MILP relaxation to determine which NLP subproblem to solve by \emph{fixing} the discrete decisions of MINLP problem, restricting it to a hyperrectangle of the domain (Section~\ref{sec:fixing}).
Then, we \emph{exclude} said hyperrectangle from the solution space of the MILP relaxation, such that re-optimizing it will result in either a tighter bound or a certificate of optimality.
Furthermore, the new solution to the MILP relaxation will lie in a different hyperrectangle, focusing on a new NLP subproblem.
}

\subsection{MILP Relaxation}\label{sec:milp-relaxation}

\new{%
The only nonlinear terms in \ref{eq:minlp} are the non-convex constraints \eqref{eq:interpolation-bilinears} over the $\xi$-space, which come from the multilinear interpolation of the black-box functions.
The tightest possible linear relaxation for each constraint\footnote{By ``linear relaxation'' of a constraint, we mean a polyhedron that contains the set defined by the constraint.} in \eqref{eq:interpolation-bilinears} is precisely its convex hull.
This is because the set defined by each \eqref{eq:interpolation-bilinears} is the graph of a multilinear constraint, and the convex envelope of a multilinear function is  polyhedral~\citep{rikunConvexEnvelopeFormula1997,sundarPiecewisePolyhedralFormulations2021}.
}
We derive the desired MILP relaxation of \ref{eq:minlp} by exploiting the traditional SOS2 formulation for piecewise-linear functions~\citep{bealeTomlin1970,kehaModelsRepresentingPiecewise2004}\new{, and we show that a simple modification yields precisely the desired convex envelopes}.
\new{In other words, we show that the resulting MILP relaxation is as tight as possible, as fixing the discrete variables always results in the convex hull of the MINLP problem (under the same fixing values).}

\new{%
Recall that our breakpoints $\mathcal{B}$ are distributed in a rectangular grid, such that the domain variables $\bm{x}$ always lie in a hyperrectangle of consecutive breakpoints.
In a piecewise-linear \emph{interpolation} over $\mathcal{B}$, the SOS2 constraints over $\bm{\xi}$ are used to express $\bm{x}$ as a CC of the hyperrectangle's vertices, exactly as in our MINLP formulation \eqref{eq:minlp} of the multilinear interpolants $\widetilde{f}$.
However, in the piecewise-linear case, $\bm{x}$ is further constrained to be a CC of \emph{at most} $n$ vertices of any given hyperrectangle of the domain, effectively further partitioning the domain into simplexes~\citep{misenerPiecewiseLinearApproximationsMultidimensional2010}.
Figure~\ref{fig:piecewise-interp-example} illustrates the piecewise-linear interpolation for the case of a function over two variables.
Note that the domain is naturally partitioned in rectangles by the breakpoints, then it is further subdivided into triangles (2D simplexes) whose vertices are the $n$ breakpoints used in the CC.
}

\begin{figure}[h]
    \centering
    \begin{subfigure}[t]{0.49\textwidth}
    \centering
    \begin{tikzpicture}
        \begin{axis}[
            axis lines = center,
            point meta min=0.5,
            point meta max=2,
            view={15}{30},
            xmin=0.5,
            xmax=2.5,
            xtick={1,2},
            xticklabels={$\hat{x}^0$,$\hat{x}^1$},
            ymin=0.5,
            ymax=2.5,
            ytick={1,2},
            yticklabels={$\hat{y}^0$,$\hat{y}^1$},
            y tick label style={anchor=south},
            zlabel=$\widetilde{f}$,
            zmin=0,
            zmax=2,
            ztick={0},
            z label style={anchor=east},
	]
            \addplot3 [black,dashed] coordinates {
                (1,0.5,0) (1,2.5,0)
                
                (2,0.5,0) (2,2.5,0)
                
                (0.5,1,0) (2.5,1,0)
                
                (0.5,2,0) (2.5,2,0)
		
                (1,1,0) (2,2,0)
            };
            \addplot3[patch,patch refines=0,shader=faceted interp]  coordinates {
                (1,1,2)
                (1,2,0.5)
                (2,2,2)
                
                (1,1,2)
                (2,1,0.5)
                (2,2,2)
            };
            \addplot3 [black,dotted,mark=*,mark size=1pt] coordinates {(1,1,0) (1,1,2)};
            \addplot3 [black,dotted,mark=*,mark size=1pt] coordinates {(2,2,0) (2,2,2)};
            \addplot3 [black,dotted,mark=*,mark size=1pt] coordinates {(2,1,0) (2,1,0.5)};
            \addplot3 [black,dotted,mark=*,mark size=1pt] coordinates {(1,2,0) (1,2,0.5)};
        \end{axis}
    \end{tikzpicture}
    \caption{}\label{fig:pwl-interp-a}
    \end{subfigure}
    \begin{subfigure}[t]{0.49\textwidth}
    \centering
    \begin{tikzpicture}
        \begin{axis}[
            axis lines = center,
            point meta min=0.5,
            point meta max=2,
            view={15}{30},
            xmin=0.5,
            xmax=2.5,
            xtick={1,2},
            xticklabels={$\hat{x}^0$,$\hat{x}^1$},
            ymin=0.5,
            ymax=2.5,
            ytick={1,2},
            yticklabels={$\hat{y}^0$,$\hat{y}^1$},
            y tick label style={anchor=south},
            zlabel=$\widetilde{f}$,
            zmin=0,
            zmax=2,
            ztick={0},
            z label style={anchor=east},
	]
            \addplot3 [black,dashed] coordinates {
                (1,0.5,0) (1,2.5,0)
                
                (2,0.5,0) (2,2.5,0)
                
                (0.5,1,0) (2.5,1,0)
                
                (0.5,2,0) (2.5,2,0)
		
                (1,2,0) (2,1,0)
            };

            \addplot3[patch,patch refines=0,shader=faceted interp]  coordinates {
                (1,2,0.5)
                (2,2,2)
                (2,1,0.5)

                (1,1,2)
                (1,2,0.5)
                (2,1,0.5)
            };

            \addplot3 [black,dotted,mark=*,mark size=1pt] coordinates {(1,1,0) (1,1,2)};
            \addplot3 [black,dotted,mark=*,mark size=1pt] coordinates {(2,2,0) (2,2,2)};
            \addplot3 [black,dotted,mark=*,mark size=1pt] coordinates {(2,1,0) (2,1,0.5)};
            \addplot3 [black,dotted,mark=*,mark size=1pt] coordinates {(1,2,0) (1,2,0.5)};
        \end{axis}
    \end{tikzpicture}
    \caption{}\label{fig:pwl-interp-b}
    \end{subfigure}
    \caption{\new{Example of piecewise-linear interpolation with the same look-up table but different triangulations of the domain.}}
    \label{fig:piecewise-interp-example}
\end{figure}

\new{%
We derive a piecewise-linear \emph{relaxation} by removing the constraint that the domain variable must be a CC of at most $n$ consecutive breakpoints, i.e., allowing it to be expressed as a convex combination of all $2n$ consecutive breakpoints.
This allows the interpolant variable $\widetilde{f}$ to assume any value between all possible piecewise-linear interpolations, thus, the convex envelope.
Fig.~\ref{fig:pwl-relaxation} illustrates the piecewise-linear relaxation of the function illustrated in Fig.~\ref{fig:piecewise-interp-example}.
Note how the facets of the convex envelope in Fig.~\ref{fig:pwl-relaxation} match exactly the graphs of the interpolants in Fig.~\ref{fig:pwl-interp-a} and \ref{fig:pwl-interp-b}.
}

\new{%
Since our multilinear interpolation is already formulated through SOS2 constraints, we need only to replace constraints \eqref{eq:interpolation-bilinears} by
\begin{equation}\label{eq:relaxed-interpolation-bilinears}
     \xi^j_{\hat{k}_j} = \sum_{\substack{\bm{k} \in \mathcal{I}: \\ k_j = \hat{k}_j}} \lambda_{\bm{k}},\quad \forall \hat{k}_j\in \mathcal{I}_j,\, j=1,\ldots,n .
\end{equation}
Then, our MILP relaxation of \ref{eq:minlp} is
}
\begin{equation}\label{eq:milp}\tag{$\overline{\mathcal{P}}$ }
\begin{aligned}
    \overline{C} = \min_{\bm{\xi},\bm{\lambda},\bm{y}} \quad & \widetilde{f}_0 + \bm{c}^{T} \bm{y}  \\
    \textrm{s.t.} \quad & \widetilde{f}_i + \bm{a}_i^{T} \bm{y} \le b_i, & i=1,\ldots,m \\
			& \widetilde{f}_i = \sum_{\bm{k}\in\mathcal{I}} \lambda_{\bm{k}}  f_i(\hat{\bm{x}}^{\bm{k}}), & i=0,\ldots,m, \\
        &\xi^j_{\hat{k}_j} = \sum_{\substack{\bm{k} \in \mathcal{I}: \\ k_j = \hat{k}_j}} \lambda_{\bm{k}},& \forall \hat{k}_j\in \mathcal{I}_j,\, j=1,\ldots,n, \\
   &x_{j} = \sum_{k\in\mathcal{I}_j} \xi^j_{k} \hat{x}_{j}^{k}, &j=1,\ldots,n, \\
   &1 = \sum_{k\in\mathcal{I}_j} \xi^j_{k}, &j=1,\ldots,n, \\
      & \bm{\xi^j} \in \text{SOS2}(\mathcal{I}_j), & j=1,\ldots,n, \\
      & \bm{\lambda} \in [0,1]^{|\mathcal{I}|}, \\
      & \bm{x}\in X, \,\bm{y}\in \left\{ 0,1 \right\}^{p} 
.\end{aligned}
\end{equation}
\new{%
Differently from problem \ref{eq:minlp}, $\bm{\lambda}$ is not functionally dependent on $\bm{\xi}$ in \ref{eq:milp}, that is, \ref{eq:milp} is effectively a problem over the $(\bm{\xi},\bm{\lambda},\bm{y})$ extended space.
Furthermore, the number of $\lambda$ variables grows exponentially with $n$, as there must be one $\lambda$ variable for each element in $\mathcal{B} = \mathcal{B}_1\times \cdots \times \mathcal{B}_n$. However, at most $2^n$ of them are non-null in any feasible solution, once the SOS2 constraints over the $\xi$ variables force every $\lambda$ variable that is not a vertex of the selected hyperrectangle to be 0.
}

\new{%
Problem \ref{eq:milp} contains the tightest possible relaxation of the multilinear constraints of \ref{eq:minlp}, given that each $\lambda_{\bm{k}}$ variable must be a convex combination of the product of the $\xi^1_{k_1},\ldots,\xi^n_{k_n}$ variables at their extremes~\citep{rikunConvexEnvelopeFormula1997,sundarPiecewisePolyhedralFormulations2021}.
}
For completeness, we demonstrate through Proposition~\ref{prop} that the solution space of \ref{eq:milp} is indeed a relaxation of the solution space of \ref{eq:minlp}, which is a necessary condition for the correctness of Algorithm~\ref{alg:algorithm}.

\begin{proposition}\label{prop}
    Problem \ref{eq:milp} is a relaxation of \ref{eq:minlp}.
\end{proposition}
\begin{proof}
    Because the only change from \ref{eq:minlp} to \ref{eq:milp} is the removal of constraint \eqref{eq:interpolation-bilinears} and the addition of constraint \eqref{eq:relaxed-interpolation-bilinears}, it suffices to show that any solution of \ref{eq:minlp} satisfies \eqref{eq:relaxed-interpolation-bilinears}.
    Thus, let $( \widetilde{\bm{x}}, \widetilde{\bm{y}}, \widetilde{\bm{\xi}},\widetilde{\bm{\lambda}} ) $ be a feasible solution of \ref{eq:minlp}.
    \new{%
    Consider first the case $j=1$ and take any $\hat{k}_1\in\mathcal{I}_1$.
    Because $\widetilde{\lambda}_{\bm{k}}$ respects \eqref{eq:interpolation-bilinears} for any $\bm{k}\in\mathcal{I}$, we have
    \begin{equation}
    \begin{aligned}
        \sum_{\substack{\bm{k} \in \mathcal{I}: \\ k_1 = \hat{k}_1}} \widetilde{\lambda}_{\bm{k}}
            &= \sum_{\substack{\bm{k} \in \mathcal{I}: \\ k_1 = \hat{k}_1}} \prod_{j=1}^{n} \widetilde{\xi}^j_{k_j} \\
            &= \widetilde{\xi}^1_{\hat{k}_1} \sum_{k_2 \in \mathcal{I}_2} \widetilde{\xi}^2_{k_2} \sum_{k_3 \in \mathcal{I}_3} \widetilde{\xi}^3_{k_3} \cdots \sum_{k_n \in \mathcal{I}_n}  \widetilde{\xi}^n_{k_n} \\
            &= \widetilde{\xi}^1_{\hat{k}_1},
    \end{aligned}
    \end{equation}
    where the last equality comes from $\widetilde{\bm{\xi}}$ satisfying
    \begin{equation}
	\sum_{k_j\in \mathcal{I}_j} \widetilde{\xi}^j_{k_j} = 1,\,j=1,\ldots,n
    \end{equation}
    as it is feasible for \ref{eq:minlp}.
    Therefore, the first constraint of \eqref{eq:relaxed-interpolation-bilinears} is satisfied, given that the choice for $\hat{k}_1$ was arbitrary.
    \emph{Mutatis mutandis}, the same can be shown for $j=2,\ldots,n$, implying that all constraints in \eqref{eq:relaxed-interpolation-bilinears} are satisfied by a feasible solution of \ref{eq:minlp}.
    }
\end{proof}

One interesting characteristic of the proposed piecewise relaxation is that it can also be used to relax multilinear constraints.
For example, suppose that function $f_m$ in problem \eqref{eq:problem} is not a black-box function, but rather a known function of the form
\begin{equation}
    f_m(\bm{x}) = \sum_{(i,j)} a_{ij}x_i x_j
.\end{equation}
\new{Then, its multilinear interpolation is exact and} the relaxation of $f_m$ in \ref{eq:milp} will match precisely the piecewise McCormick envelopes~\citep{CASTRO2015300} at the partitioning induced by the breakpoints of the look-up table, since that also characterizes the convex hull in such cases~\citep{al-khayyalJointlyConstrainedBiconvex1983}.

\subsection{Fixing}\label{sec:fixing}

\new{%
In our algorithm, we find candidate solutions by solving NLP subproblems that stem from fixing all discrete decisions of \ref{eq:minlp}.
This is a usual step in mixed-integer programming which consists of fixing the assignment of the integer variables.
In \ref{eq:minlp}, the SOS2 constraints make this process more involved, since the discrete decisions are not explicit.
}

\new{%
We begin by observing that the SOS2 constraints $\bm{\xi}^j\in \text{SOS2}(\mathcal{I}_j)$ can be equivalently formulated as
\begin{equation}\label{eq:sos2-binaries}
    \bm{\xi}^j \le \bm{\zeta}^{j} \text{ and }  \bm{\zeta}^{j}\in \text{SOS2}(\mathcal{I}_j)
,\end{equation}
for $\bm{\zeta}^{j} \in \{0,1\}^{|\mathcal{B}_j|}$ and $j=1,\ldots,n$.
Each $\bm{\zeta}^{j}$ can be understood as the discrete decisions associated with the SOS2 constraint over $\bm{\xi}^j$, that is, the selection of the pair of consecutive breakpoints along the $j$-th axis.
}

\new{%
Now, take $\overline{\bm{y}} \in \{0,1\}^{p}$ and $\overline{\bm{\zeta}}^j\in \text{SOS2}(\mathcal{I}_j) \cap \{0,1\}^{|\mathcal{B}_j|}$, for every $j=1,\ldots,n$.
Then, it is easy to see that fixing $\bm{y}=\overline{\bm{y}}$ and $\bm{\xi}^j \le \bm{\zeta}^{j},\,j=1,\ldots,n$, reduces \ref{eq:minlp} to an NLP.
More formally, let $P$ denote the feasible region of \ref{eq:minlp}, and define
\begin{equation}\label{eq:nlp-subproblem}
    P(\bm{\overline{\zeta}},\bm{\overline{y}}) \triangleq \left\{ (\bm{\xi},\bm{y}) \in P : \bm{y}=\overline{\bm{y}},\, \bm{\xi}^j\le \overline{\bm{\zeta}}^j,\, j=1,\ldots,n \right\}
\end{equation}}
\new{%
If each $\overline{\bm{\zeta}}^j\in \text{SOS2}(\mathcal{I}_j)$, then at most two dimensions of $\bm{\xi}^j$ are \emph{not} fixed to $0$ in $P(\bm{\overline{\zeta}},\bm{\overline{y}})$, i.e., the dimensions over which $\bm{\xi}^j$ can assume positive value.
   Further, by projecting $P(\bm{\overline{\zeta}},\bm{\overline{y}})$ back onto the $x$-space through $x_{j} = \sum_{k\in\mathcal{I}_j} \xi^j_{k} \hat{x}_{j}^{k}$, we see that $x_j\in [\hat{x}_j^{\overline{k}},\hat{x}_j^{\overline{k}+1}]$, where $\overline{k}$ is the index of the first positive value of $\overline{\bm{\zeta}}^j$.
In other words, our fixing strategy is equivalent to restricting each $x_j$ to lie between two consecutive breakpoints in $\mathcal{B}_j$, thus, restricting $\bm{x}$ to lie in a single hyperrectangle of the partitioning induced by the breakpoints in $\mathcal{B}$ of the functions' domain.
}

\new{%
Now let $(\widetilde{\bm{\xi}},\widetilde{\bm{y}})\in P$ be a feasible solution for \ref{eq:minlp}. 
We can define an NLP subproblem for which $(\widetilde{\bm{\xi}},\widetilde{\bm{y}})$ is feasible by selecting a binary vector $\widetilde{\bm{\zeta}}^j\in\text{SOS2}(\mathcal{I}_j)$ such that $\widetilde{\bm{\xi}}^j\le\widetilde{\bm{\zeta}}^j$, $j=1,\ldots,n$.
However, $\widetilde{\bm{\zeta}} = \lceil \widetilde{\bm{\xi}}\rceil$ satisfies this condition.
Thus,
\begin{equation}\label{eq:rounding}
(\widetilde{\bm{\xi}},\widetilde{\bm{y}})\in P(\lceil \widetilde{\bm{\xi}}\rceil,\widetilde{\bm{y}}), \, \forall (\widetilde{\bm{\xi}},\widetilde{\bm{y}})\in P
.\end{equation}
Further, let $\overline{P}$ be the feasible region of our MILP relaxation \ref{eq:milp}, and take a feasible solution $(\overline{\bm{\xi}}, \overline{\bm{\lambda}}, \overline{\bm{y}}) \in \overline{P}$.
Then, a consequence of \eqref{eq:rounding} is that
\begin{equation}
(\overline{\bm{\xi}},\overline{\bm{y}}) \in P\iff (\overline{\bm{\xi}},\overline{\bm{y}})\in P(\lceil \overline{\bm{\xi}}\rceil,\overline{\bm{y}})
.\end{equation}
}

\new{%
Recall that our goal is to fix all discrete decisions of our MINLP subproblem.
The rounding procedure in \eqref{eq:rounding} provides a way to fix the discrete decisions in the SOS2 constraints without formulating them as binary variables. 
In other words, we do \emph{not} need additional binary variables (e.g., $\bm{\zeta}$) to formulate the NLP subproblems. 
The fixing operation is illustrated in Figure~\ref{fig:alg-b}, with $P$, the feasible regions of \ref{eq:minlp}, and ${P}(\lceil\, \overline{\bm{\xi}} \,\rceil ,\overline{\bm{y}})$, the feasible region of one of the NLP subproblems.
}

\subsection{Excluding}\label{sec:excluding}

\new{%
Suppose we solve our MILP relaxation \ref{eq:milp} and find an optimal solution $(\overline{\bm{\xi}},\overline{\bm{\lambda}},\overline{\bm{y}}) \in \overline{P}$ with cost $\overline{C}$.
We then use fixing and solve the associated NLP subproblem, finding an optimum $(\widetilde{\bm{\xi}},\widetilde{\bm{y}}) \in P(\lceil\,\overline{\bm{\xi}} \,\rceil,\overline{\bm{y}})\subseteq P$, with cost $C$.
If $\overline{C} = C$, then we have a certificate of optimality for our candidate solution $(\widetilde{\bm{\xi}},\widetilde{\bm{y}})$.
On the other hand, if $\overline{C} < C$, better solutions for \ref{eq:minlp} might only exist in $P \setminus P(\lceil\,\widetilde{\bm{\xi}} \,\rceil,\widetilde{\bm{y}})$, since our candidate is optimal for the NLP subproblem.
Therefore, a tighter dual bound than $\overline{C}$ can be computed by \emph{excluding} $P(\lceil\,\widetilde{\bm{\xi}} \,\rceil,\widetilde{\bm{y}})$ from the feasible region of our MILP relaxation.
}

\new{%
Let $\overline{P}(\lceil\,\overline{\bm{\xi}} \,\rceil,\overline{\bm{y}})$ be defined with respect to $\overline{P}$ akin to how $P(\lceil\,\overline{\bm{\xi}} \,\rceil,\overline{\bm{y}})$ is defined with respect to $P$ in \eqref{eq:nlp-subproblem}, that is, through the imposition of the same fixing constraints.
Then, we have
\begin{equation}
P \setminus P(\lceil\,\overline{\bm{\xi}} \,\rceil,\overline{\bm{y}}) \subseteq \overline{P} \setminus \overline{P}(\lceil\,\overline{\bm{\xi}} \,\rceil,\overline{\bm{y}}) 
.\end{equation}
So instead of excluding $P(\lceil\,\widetilde{\bm{\xi}} \,\rceil,\widetilde{\bm{y}})$ from our MILP relaxation $\overline{\mathcal{P}}$, we can exclude $\overline{P}(\lceil\,\overline{\bm{\xi}} \,\rceil,\overline{\bm{y}})$.
}

\new{%
A common way of excluding a given discrete decision from the solution space is through a no-good cut~\citep{DAMBROSIO2010341,balas_jeroslaw_72}.
Extending such a combinatorial cut to the SOS2 variables results in the strict inequality
\begin{equation}\label{eq:cut}
    \sum_{j=1}^{n} \sum_{k\in \mathcal{I}_j} \left( 1-\left\lceil \,\overline{\xi}^j_{k}\, \right\rceil \right)  \xi^j_{k}   + \sum_{l=1}^{p} \Bigl \{ \overline{y}_l (1-y_l) + (1-\overline{y}_l) y_l\Bigr \} > 0
,\end{equation}
which indeed implies precisely the desired result.
However, imposing a strict inequality may lead to unreachable optima and numerical instabilities.
}

\new{%
We propose a weaker cut based on the extended formulation of the SOS2 constraints using additional binary variables.
Consider the extended formulation of the SOS2 constraints in \eqref{eq:sos2-binaries}.
Then, the no-good cut
\begin{equation}\label{eq:no-good-cut}
    \sum_{j=1}^{n} \sum_{k\in \mathcal{I}_j} \left(1 - \lceil\, \overline{\xi}^j_{k} \,\rceil \right) \zeta^j_{k}  + \sum_{l=1}^{p} \Bigl \{ \overline{y}_l (1-y_l) + (1-\overline{y}_l) y_l\Bigr \} \ge 1
\end{equation}
excludes \emph{at least} the interior of $\overline{P}(\lceil\,\overline{\bm{\xi}} \,\rceil,\overline{\bm{y}})$ from our MILP relaxation.
Note that, due to the SOS2 constraints over each $\bm{\zeta}^j$, we restrict the cut to the indices $k\in\mathcal{I}_j$ for which $\overline{\bm{\zeta}}^j_k = 0$.
Nevertheless, differently from the fixing step, the addition of cut \eqref{eq:no-good-cut} to our MILP relaxation \emph{requires} the extended formulation of the SOS2 constraints, that is, the addition of the $\zeta$ binary variables to \ref{eq:milp}.
}

\new{%
To see the effect of adding cut \eqref{eq:no-good-cut} to the MILP relaxation \ref{eq:milp}, note that its feasible region can be expressed as
\begin{equation}\label{eq:union-subproblems}
    \overline{P} = \bigcup_{\substack{\bm{y}\in\{0,1\}^p,\\ \bm{\zeta}^j\in \text{SOS2}(\mathcal{I}_j), j=1,\ldots,n}} \overline{P}(\bm{\zeta},\bm{y})
,\end{equation}
which is in line with the classic understanding of the feasible region of MILP problems as the union of finitely many polyhedra.
Thus, the no-good cut excludes $\overline{P}(\lceil\,\overline{\bm{\xi}} \,\rceil,\overline{\bm{y}})$ from the union.
Because there might be other polyhedra that intersect $\overline{P}(\lceil\,\overline{\bm{\xi}} \,\rceil,\overline{\bm{y}})$ at its border, the cut is only guaranteed to exclude the interior from the feasible region.
However, iteratively solving the relaxation and adding such cuts is guaranteed to eventually exclude the entire solution space.
}

\new{%
Finally, we note that most formulations for the SOS2 constraints found in the literature~\citep{vielma2010modeling,Silva:EJOR:2014} can be used with only minor modifications to cut \eqref{eq:no-good-cut}.
}

\subsection{The Algorithm}

The proposed relax-fix-and-exclude approach is formalized in Algorithm~\ref{alg:algorithm}.
Our algorithm progressively explores the solution space of the MINLP problem \ref{eq:minlp} guided by the solutions of its MILP relaxation \ref{eq:milp}.
\new{%
An optimal solution to \ref{eq:milp} gives a dual bound, which we use as a stopping criterion.
If the criterion has not been met, we construct the NLP subproblem by \emph{fixing} the binary and SOS2 variables.
An optimal solution to the NLP subproblem yields a candidate solution and, possibly, a tighter primal bound.
Then, the hyperrectangle containing the feasible region of the NLP subproblem is \emph{excluded} from the MILP relaxation.
An optimal solution to the updated MILP relaxation yields a tighter dual bound and a new hyperrectangle to explore.
}
Figure~\ref{fig:alg-steps} illustrates the main steps of the algorithm.
A proof of the algorithm's convergence is established through Theorem~\ref{theo}.

\new{%
The algorithm's guarantees rely on the availability of a global solver for the NLP subproblem.
The fact that the solution found at Line~\ref{line:fix} is a global optimum of the subproblem is what allows us to exclude $\overline{P}(\lceil\, \overline{\bm{\xi}} \,\rceil,\overline{\bm{y}})$ from our MILP relaxation and, thus, guide our exploration of the search space and generate tighter dual bounds.
We note, however, that this is not such a strong assumption, as although the NLP subproblem may be nonconvex, it is smooth and bounded.
In fact, many contemporary solvers (e.g., Gurobi~\cite{gurobi}, BARON~\cite{sahinidisBARONGeneralPurpose1996,zhangSolvingContinuousDiscrete2024}, and SCIP~\cite{BolusaniEtal2024OO,BolusaniEtal2024ZR}) already provide global optimality guarantees for these types of problems, exploiting spatial branch-and-bound and domain reduction techniques~\cite{puranikDomainReductionTechniques2017a}.
}

\begin{theorem}\label{theo}
    Algorithm~\ref{alg:algorithm} terminates in a finite number of steps and returns either (i) $\infty$, if the problem is infeasible, (ii) $-\infty$ if the problem is unbounded, or (iii) an optimal solution to \ref{eq:minlp}.
\end{theorem}

\begin{proof}
\new{%
First, we note that the feasible region of the NLP subproblems define an exact cover for the feasible region of \ref{eq:minlp}.
More precisely, by the definition of $P(\bm{\zeta},\bm{y})$ in \eqref{eq:nlp-subproblem}, we have, as already presented in Section \ref{sec:excluding},
\begin{equation}
P = \bigcup_{\substack{\bm{y}\in\{0,1\}^p,\\ \bm{\zeta}^j\in \text{SOS2}(\mathcal{I}_j), j=1,\ldots,n}} P(\bm{\zeta},\bm{y})
,\end{equation}
which is a direct consequence of the partitioning induced in the $x$-space from the breakpoints $\mathcal{B}_1,\ldots,\mathcal{B}_n$.
To see this, recall that fixing $\bm{\zeta}$ is equivalent to restraining $\bm{x}$ to lie in a particular hyperrectangle of $\R^n$ defined by consecutive breakpoints (see Section \ref{sec:fixing}).
Similarly, the family of sets $\overline{P}(\bm{\zeta},\bm{y})$ also defines an exact cover for $\overline{P}$.
}

\new{%
The finite termination of the algorithm is a direct consequence of the exclude step at Line~\ref{line:cut}, which guarantees that the same subproblem will never be visited twice.
Putting it into terms, we have that, for any $(\overline{\bm{\xi}},\overline{\bm{\lambda}},\overline{\bm{y}})\in\overline{P}$,
\begin{equation}
(\bm{\xi},\bm{\lambda},\bm{y}) \in \overline{P} \setminus \overline{P}(\lceil\,\overline{\bm{\xi}}\,\rceil,\overline{\bm{y}}) \implies \lceil\,\bm{\xi}\,\rceil \neq \lceil\,\overline{\bm{\xi}}\,\rceil
.\end{equation}
Thus, once the multilinear interpolation implies that our feasible region is bounded, we will eventually run out of subproblems (or hyperrectangles) to explore, resulting in $\overline{P}=\emptyset$ and inducing $\overline{C}=\infty$ at Line~\ref{line:milpsolver}, which meets the termination condition.
}

\new{%
Consider now the case in which the algorithm terminates and no candidate solution is found, i.e., the algorithm returns $C = \infty$.
This implies that the termination condition of the while loop was met through $\overline{C}=\infty$.
But for any $(\overline{\bm{\xi}},\overline{\bm{\lambda}},\overline{\bm{y}})\in\overline{P}$, we have $P(\overline{\bm{\xi}},\overline{\bm{y}}) \subseteq \text{Proj}_{\bm{\xi}\bm{y}}(\overline{P}(\overline{\bm{\xi}},\overline{\bm{y}}))$ and
\begin{equation}
P \setminus P(\overline{\bm{\xi}},\overline{\bm{y}}) \subseteq \text{Proj}_{\bm{\xi}\bm{y}} (\overline{P}\setminus\overline{P}(\overline{\bm{\xi}},\overline{\bm{y}}))
,\end{equation}
where $\text{Proj}_{\bm{\xi}\bm{y}}(\cdot)$ denotes the projection in the $(\bm{\xi}$,$\bm{y})$-space.
So, once no feasible solution was found during the iterations of the while loop, we have a guarantee that $P=\varnothing$.
}

\new{%
Suppose now the algorithm terminates with a candidate solution that is not optimal.
Thus, there must exist a solution $(\bm{\xi}',\bm{y}')\in P$ with an associated cost $C' < C$.
However, once we assume the NLP solver used in the fix step (at Line \ref{line:fix}) has global optimality guarantees, $(\bm{\xi}',\bm{y}')$ must not be feasible for any of the NLP subproblems visited.
In other words, it must be true that $(\bm{\xi}',\bm{y}') \in \text{Proj}_{\bm{\xi}\bm{y}}(\overline{P}(\overline{\bm{\xi}},\overline{\bm{y}}))$, where $\overline{P}$ is the feasible region of the MILP at the termination, i.e., after all exclusions.
But then we reach a contradiction, as the termination condition implies $C' \ge \overline{C} \ge C$.
Thus, the candidate solution returned by the algorithm must be optimal.
}
\end{proof}

\begin{figure}[H]
    \centering
    \ref{labels}  
    \\
    \vspace{0.2cm}
    \begin{subfigure}[t]{0.49\textwidth}
        \centering
        \begin{tikzpicture}
    	\tikzstyle{barP} = [draw,blue,fill=blue!20,closed=true]
    	\tikzstyle{tildeP} = [draw,orange,use Hobby shortcut,closed=true,fill=orange!30]
    
    	\path [style=barP] (3,0.25) -- (4,0.3) -- (5,1.3) -- (5.3,2) -- (5.4,3) -- (5,3.8) -- (4.6,3.9) -- (4.1,3.9) -- (3.6,3.5) -- (3,2.5) -- (2.75,2) -- (2.5,1.3) -- (2.45,0.65) -- (3,0.25);
    	\path [style=barP] (2,5.85) -- (3,5.95) -- (3.9,5.25) -- (3.9,4.75) -- (3.7,4.35) -- (3,4.05) -- (2,4.15) -- (1.3,4.25) -- (0.9,4.75) -- (0.9,5.25) -- (1.5,5.75) -- (2,5.85);
    
    	\path[style=tildeP] (2.5,1) .. (4,.5) .. (5.2,2) .. (4,3.7) .. (3.8,3.4) .. (3.5,2);
    	\path[style=tildeP] (2,5.75) .. (3.8,4.75) .. (2,4.25) .. (1,4.75);
    
        \begin{axis}[
        	anchor=origin,
        	at={(0pt,0pt)},
        	disabledatascaling,
        	x=1cm,y=1cm,
        	hide axis,
        	xmin=0,
        	xmax=6,
        	ymin=0,
        	ymax=6,
            legend columns=-1,
            legend style={
                draw=none,
                column sep=0.5cm,
                /tikz/every odd column/.append style={column sep=0.1cm},
            },
            legend to name=labels,
        ]
            \addlegendimage{draw,blue,fill=blue!20,closed=true,area legend}
            \addlegendentry{$\overline{P}$}
            \addlegendimage{draw,orange,use Hobby shortcut,closed=true,fill=orange!30,area legend}
            \addlegendentry{$P$}
            \addlegendimage{draw,orange,use Hobby shortcut,closed=true,fill=orange!60,color=orange,area legend}
            \addlegendentry{$P(\lceil\, \overline{\bm{\xi}} \,\rceil,\overline{\bm{y}})$}
            
            \addlegendimage{legend image with text=$\overline{\star}$}
            \addlegendentry{$(\overline{\bm{\xi}},\overline{\bm{\lambda}},\overline{\bm{y}})$}
            \addlegendimage{legend image with text=$\widetilde{\star}$}
            \addlegendentry{$(\widetilde{\bm{\xi}},\widetilde{\bm{y}})$}
    
    	\foreach \x in {2,4} {
    	    \addplot [black,dashed] coordinates {
    		(\x,0) (\x,8)
    
    		(0,\x) (8,\x)
    	    };
    	}
        \end{axis}
    
    
        \node (barPstar) at (5.4,3.05) {$\overline{\star}$};
        \end{tikzpicture}
        \caption{Before the first iteration, a solution $\overline{\star}$ for the MILP relaxation \ref{eq:milp} is computed (Line~\ref{line:first-milpsolve})}\label{fig:alg-a}
    \end{subfigure}\hfill
    \begin{subfigure}[t]{0.49\textwidth}
        \centering
        \begin{tikzpicture}
    	\tikzstyle{barP} = [draw,blue,fill=blue!20,closed=true]
    	\tikzstyle{tildeP} = [draw,orange,use Hobby shortcut,closed=true,fill=orange!30]
    
    	\path [style=barP] (3,0.25) -- (4,0.3) -- (5,1.3) -- (5.3,2) -- (5.4,3) -- (5,3.8) -- (4.6,3.9) -- (4.1,3.9) -- (3.6,3.5) -- (3,2.5) -- (2.75,2) -- (2.5,1.3) -- (2.45,0.65) -- (3,0.25);
    	\path [style=barP] (2,5.85) -- (3,5.95) -- (3.9,5.25) -- (3.9,4.75) -- (3.7,4.35) -- (3,4.05) -- (2,4.15) -- (1.3,4.25) -- (0.9,4.75) -- (0.9,5.25) -- (1.5,5.75) -- (2,5.85);
    
    	\path[style=tildeP] (2.5,1) .. (4,.5) .. (5.2,2) .. (4,3.7) .. (3.8,3.4) .. (3.5,2);
    	\path[style=tildeP] (2,5.75) .. (3.8,4.75) .. (2,4.25) .. (1,4.75);
    
        \begin{axis}[
    	anchor=origin,
    	at={(0pt,0pt)},
    	disabledatascaling,
    	x=1cm,y=1cm,
    	hide axis,
    	xmin=0,
    	xmax=6,
    	ymin=0,
    	ymax=6,
        ]
    	\foreach \x in {2,4} {
    	    \addplot [black,dashed] coordinates {
    		(\x,0) (\x,8)
    
    		(0,\x) (8,\x)
    	    };
    	}
        \end{axis}
    
        \begin{scope} 
    	\path [clip] (4,2) -- (6,2) -- (6,4) -- (4,4) -- (4,2);
    	\path [style=tildeP,color=orange,fill=orange!60] (2.5,1) .. (4,.5) .. (5.2,2) .. (4,3.7) .. (3.8,3.4) .. (3.5,2);
        \end{scope}
    
        \node (barPstar) at (5.4,3.05) {$\overline{\star}$};
    
        \node (tildePstar) at (4.3,3.4) {$\widetilde{\star}$};
    
        \end{tikzpicture}
        \caption{In the loop, at Line~\ref{line:fix}, the NLP subproblem over $P(\lceil\, \overline{\bm{\xi}} \,\rceil,\overline{\bm{y}})$ is constructed by \emph{fixing} the integer variables of \ref{eq:minlp} with the values of \ref{eq:milp}. The subproblem is solved to optimality ($\widetilde{\bm{x}}$), which we note results in a feasible solution for \ref{eq:minlp}.}\label{fig:alg-b}
    \end{subfigure}
    \begin{subfigure}[t]{0.49\textwidth}
        \centering
        \begin{tikzpicture}
    	\tikzstyle{barP} = [draw,blue,fill=blue!20,closed=true]
    	\tikzstyle{tildeP} = [draw,orange,use Hobby shortcut,closed=true,fill=orange!30]
    
    	\path [style=barP] (3,0.25) -- (4,0.3) -- (5,1.3) -- (5.3,2) -- (5.4,3) -- (5,3.8) -- (4.6,3.9) -- (4.1,3.9) -- (3.6,3.5) -- (3,2.5) -- (2.75,2) -- (2.5,1.3) -- (2.45,0.65) -- (3,0.25);
    	\path [style=barP] (2,5.85) -- (3,5.95) -- (3.9,5.25) -- (3.9,4.75) -- (3.7,4.35) -- (3,4.05) -- (2,4.15) -- (1.3,4.25) -- (0.9,4.75) -- (0.9,5.25) -- (1.5,5.75) -- (2,5.85);
    
    	\path[style=tildeP] (2.5,1) .. (4,.5) .. (5.2,2) .. (4,3.7) .. (3.8,3.4) .. (3.5,2);
    	\path[style=tildeP] (2,5.75) .. (3.8,4.75) .. (2,4.25) .. (1,4.75);
    
        \begin{axis}[
        	anchor=origin,
        	at={(0pt,0pt)},
        	disabledatascaling,
        	x=1cm,y=1cm,
        	hide axis,
        	xmin=0,
        	xmax=6,
        	ymin=0,
        	ymax=6,
        ]
    	\foreach \x in {2,4} {
    	    \addplot [black,dashed] coordinates {
    		(\x,0) (\x,8)
    
    		(0,\x) (8,\x)
    	    };
    	}
        \end{axis}
    
        \begin{scope} 
    	\path [clip] (4,2) -- (6,2) -- (6,4) -- (4,4) -- (4,2);
    	\path [draw,white,fill=white,opacity=100] (3,0.25) -- (4,0.3) -- (5,1.3) -- (5.3,2) -- (5.4,3) -- (5,3.8) -- (4.6,3.9) -- (4.1,3.9) -- (3.6,3.5) -- (3,2.5) -- (2.75,2) -- (2.5,1.3) -- (2.45,0.65) -- (3,0.25);
    	\path [draw,blue,dashed] (3,0.25) -- (4,0.3) -- (5,1.3) -- (5.3,2) -- (5.4,3) -- (5,3.8) -- (4.6,3.9) -- (4.1,3.9) -- (3.6,3.5) -- (3,2.5) -- (2.75,2) -- (2.5,1.3) -- (2.45,0.65) -- (3,0.25);
    	\path [style=tildeP,color=orange,fill=orange!60] (2.5,1) .. (4,.5) .. (5.2,2) .. (4,3.7) .. (3.8,3.4) .. (3.5,2);
        \end{scope}
    
        \node (barPstar) at (5.4,3.05) {\color{red}$\overline{\star}$};
        \node (tildePstar) at (4.3,3.4) {$\widetilde{\star}$};
    
        \end{tikzpicture}
        \caption{The exclude step (Line~\ref{line:cut}) adds a cut to \ref{eq:milp} that removes the interior of the hyperrectangle containing $\overline{P}(\lceil\, \overline{\bm{\xi}}^\star \,\rceil,\overline{\bm{y}}^\star)$.
        Note that \ref{eq:milp} is now an MILP relaxation over the ``unexplored'' region of \ref{eq:minlp}.
        }\label{fig:alg-c}
    \end{subfigure}\hfill
    \begin{subfigure}[t]{0.49\textwidth}
        \centering
        \begin{tikzpicture}
    	\tikzstyle{barP} = [draw,blue,fill=blue!20,closed=true]
    	\tikzstyle{tildeP} = [draw,orange,use Hobby shortcut,closed=true,fill=orange!30]
    
    	\path [style=barP] (3,0.25) -- (4,0.3) -- (5,1.3) -- (5.3,2) -- (5.4,3) -- (5,3.8) -- (4.6,3.9) -- (4.1,3.9) -- (3.6,3.5) -- (3,2.5) -- (2.75,2) -- (2.5,1.3) -- (2.45,0.65) -- (3,0.25);
    	\path [style=barP] (2,5.85) -- (3,5.95) -- (3.9,5.25) -- (3.9,4.75) -- (3.7,4.35) -- (3,4.05) -- (2,4.15) -- (1.3,4.25) -- (0.9,4.75) -- (0.9,5.25) -- (1.5,5.75) -- (2,5.85);
    
    	\path[style=tildeP] (2.5,1) .. (4,.5) .. (5.2,2) .. (4,3.7) .. (3.8,3.4) .. (3.5,2);
    	\path[style=tildeP] (2,5.75) .. (3.8,4.75) .. (2,4.25) .. (1,4.75);
    
        \begin{axis}[
    	anchor=origin,
    	at={(0pt,0pt)},
    	disabledatascaling,
    	x=1cm,y=1cm,
    	hide axis,
    	xmin=0,
    	xmax=6,
    	ymin=0,
    	ymax=6,
        ]
    	\foreach \x in {2,4} {
    	    \addplot [black,dashed] coordinates {
    		(\x,0) (\x,8)
    
    		(0,\x) (8,\x)
    	    };
    	}
        \end{axis}
    
        \begin{scope} 
    	\path [clip] (4,2) -- (6,2) -- (6,4) -- (4,4) -- (4,2);
    	\path [draw,white,fill=white,opacity=100] (3,0.25) -- (4,0.3) -- (5,1.3) -- (5.3,2) -- (5.4,3) -- (5,3.8) -- (4.6,3.9) -- (4.1,3.9) -- (3.6,3.5) -- (3,2.5) -- (2.75,2) -- (2.5,1.3) -- (2.45,0.65) -- (3,0.25);
    	\path [draw,blue,dashed] (3,0.25) -- (4,0.3) -- (5,1.3) -- (5.3,2) -- (5.4,3) -- (5,3.8) -- (4.6,3.9) -- (4.1,3.9) -- (3.6,3.5) -- (3,2.5) -- (2.75,2) -- (2.5,1.3) -- (2.45,0.65) -- (3,0.25);
    	\path [style=tildeP,color=orange,fill=orange!60] (2.5,1) .. (4,.5) .. (5.2,2) .. (4,3.7) .. (3.8,3.4) .. (3.5,2);
        \end{scope}
    
        \node (barPstar) at (3.9,5.25) {$\overline{\star}$};
    
        \node (tildePstar) at (4.3,3.4) {$\widetilde{\star}$};
    
        \end{tikzpicture}
        \caption{Finally, a new optimal is computed for \ref{eq:milp} which is a lower bound of the unexplored region, and, thus, induces a stopping criteria.}\label{fig:alg-d}
    \end{subfigure}
    \caption{%
    Illustration of fixing and excluding in Algorithm~\ref{alg:algorithm}.
    The orange and blue regions represent, respectively, the feasible regions of \ref{eq:minlp} and \ref{eq:milp}, and the dashed lines represent the partitioning induced by the look-up table. Points $(\overline{\bm{x}},\overline{\bm{y}},\overline{\bm{\xi}},\overline{\bm{\lambda}})$ and $(\widetilde{\bm{x}},\widetilde{\bm{y}},\widetilde{\bm{\xi}},\widetilde{\bm{\lambda}})$ represent optimal solutions for \ref{eq:milp} and the subproblem $\mathcal{P}(\lceil\, \widetilde{\bm{\xi}} \,\rceil,\widetilde{\bm{y}})$.
    }
    \label{fig:alg-steps}
\end{figure}

\begin{algorithm}
\caption{\new{The Relax-Fix-and-Exclude algorithm. We assume the availability of global solvers both for MILP and NLP that return $-\infty$ as cost if the problem is unbounded and $\infty$ if the problem is infeasible.}}\label{alg:algorithm}
\begin{algorithmic}[1]
    \Require{Breakpoints $\mathcal{B}_j=\{\hat{x}_j^{k}\}_{k\in\mathcal{I}_j}$ for $j=1,\dots,n$, and look-up tables $\{(\hat{\bm{x}},f_i(\hat{\bm{x}})) : \hat{\bm{x}} \in \mathcal{B}_1\times\cdots\times \mathcal{B}_n\}$ for $i=0,\dots,n$.}
    
    \State $C \gets  \infty$ (infeasible)
    
    \State Solve the MILP relaxation: find an optimal solution $(\overline{\bm{\xi}},\overline{\bm{\lambda}},\overline{\bm{y}})\in \overline{P}$ with cost $\overline{C}$ \label{line:first-milpsolve}
    
    \While{$\overline{C} < C$}
        \State (Fix) Solve the NLP subproblem: find an optimal solution $(\widetilde{\bm{\xi}},\widetilde{\bm{y}})\in P(\lceil \,\overline{\bm{\xi}}\, \rceil ,\overline{\bm{y}})$ with cost $\widetilde{C}$ \label{line:fix}

    	\If{$\widetilde{C}< C$} \label{line:updateCtilde:begin}
    	    \State $C \gets \widetilde{C}$ 
    	    \State $(\bm{\xi}^\star, \bm{y}^{\star}) \gets (\widetilde{\bm{\xi}},\widetilde{\bm{y}})$ 
        \EndIf \label{line:updateCtilde:end}
    
    	\State (Exclude) Add cut \eqref{eq:no-good-cut} to $\overline{P}$ using $\lceil \,\overline{\bm{\xi}}\, \rceil$ and $\overline{\bm{y}}$   \label{line:cut}
        
        \State Solve the restricted MILP relaxation: update the optimal solution $(\overline{\bm{\xi}},\overline{\bm{\lambda}},\overline{\bm{y}})\in \overline{P}$ and the cost $\overline{C}$ \label{line:milpsolver}
    \EndWhile

    \If{$|C| < \infty$}{
    	\Return{$({\bm{\xi}}^\star, {\bm{y}}^{\star})$}
    }
    \Else{
    	\Return{$C$}
    }
    \EndIf
\end{algorithmic}
\end{algorithm}

In Algorithm \ref{alg:algorithm}, the relax-fix-and-exclude approach for the MINLP problem has three potential outcomes: 
\begin{enumerate}
\item  \textit{Infeasibility} (\(\infty\)): \new{%
The algorithm outputs \(\infty\) if either the MILP relaxation is infeasible\footnote{\new{Assuming that $\overline{C} < C$ evaluates to false when both variables take $\infty$.}}, or all NLP subproblems are infeasible.
Because the union of the feasible regions of the NLP subproblems are precisely the feasible region of \ref{eq:minlp},  the algorithm returns \(\infty\) if, and only if, the original MINLP is infeasible.
}

\item  \textit{Optimal Solution}: An optimal solution is reached when the lower bound $\overline{C}$ \new{from the MILP relaxation} converges to \new{(or surpasses) the cost} $C$ of the best-known feasible solution (the incumbent). This convergence implies that no further improvement is possible, confirming the incumbent as the global optimum of the problem.

\item  \textit{Unbounded Solution} (\(-\infty\)): If, at any point, the NLP solution within a fixed region reveals an unbounded solution (where \(C = -\infty\)), the algorithm returns \(-\infty\), signifying that the original MINLP problem is unbounded. This outcome occurs when certain regions yield solutions that are unbounded in the objective due to either the problem structure or insufficient constraints.
\end{enumerate}


\section{Oil Production Optimization} \label{sec:oil-production-optimization}

Oil Production Optimization (OPO) \citep{foss_jenson} is essential for maximizing economic gains in the petroleum industry while adhering to operational and regulatory constraints.
  Short-term production optimization focuses on control decisions with a planning horizon from hours to days, manipulating the production system to achieve an optimal steady-state operation \citep{Muller2022}.
Such optimization problems involve the complexities of fluid flow physics and the subtleties of various artificial-lifting technologies (such as gas-lift systems and electrical submersible pumps) while addressing the discrete decisions of having multiple sources \citep{CODAS2012222}.

\subsection{Problem Statement}

The problem's objective rests on maximizing the oil production of a subsea system consisting of sources (oil wells), connections (flowlines, pipelines, and manifolds), and a single sink (offshore platform).
Sources and sink establish pressure references: reservoir pressure, in the case of wells, and separator pressure at the platform on the topside.
Choke valves and lift-gas injection\footnote{In this work, lift-gas injection is the only artificial lifting method modeled.} increase and decrease (resp.) the pressure at specific network points.
The operators aim to maximize oil production by determining which wells will be open and which will be closed, the optimal opening of the valves, and the rate of lift-gas injected into the oil wells, while honoring flow assurance constraints and maintaining network balance~\citep{CODAS2012222}.
We refer the reader to \citeauthor{Muller2022}\cite{Muller2022} for a recent work on short-term production optimization of complex production systems, considering flow assurance constraints, flow routing decisions, and multiple artificial lifting techniques, among other features.

Figure \ref{fig:esquema_plataforma} depicts an offshore platform that produces  from satellite wells (connected directly to the platform) and wells connected to a subsea manifold, from which the mixed flow is transferred by risers to the platform. At the topside, the production is processed by separators that split the flow into streams of i) water treated before being reinjected or discarded, ii) oil transferred to onshore terminals by tankers, and iii) gas which is exported by subsea pipelines, used for artificial lifting or consumed by turbo generators.

\begin{figure}[!htb]
    \centering
    \includegraphics[width=\textwidth]{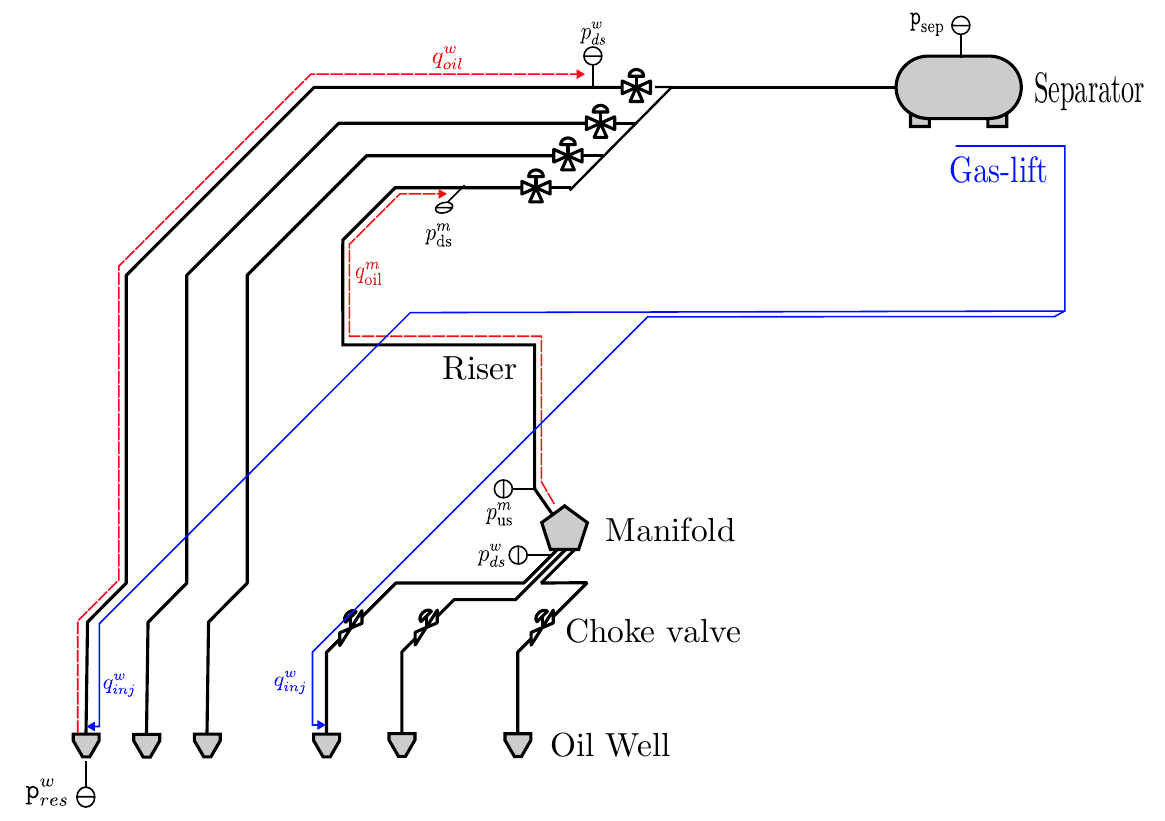}
    \caption{Platform diagram based on \citeauthor{Muller2022}\citep{Muller2022}. The platform coordinates the production from satellite wells, which flow directly to the platform, and wells connected to a subsea manifold, which collects the production and then sends it to the platform through riser pipelines. At the platform, the mixed production is processed by separators that split the production flow into water, oil, and gas streams.}
    \label{fig:esquema_plataforma}
\end{figure}

\subsection{Vertical Lift Performance (VLP) curve}

The Vertical Lift Performance (VLP) curve is a fundamental tool used to model the pressure drop across various components in an oil production system.
It describes the relationship between the flow rate of fluids and the pressure required to transport them through these components.
The VLP curve encapsulates the complex interactions of multiphase fluid flow, accounting for factors like fluid properties, gas-oil ratio (GOR), water cut (WCT), component geometry, and operational conditions.
For this reason, it is often impractical to model VLP curves analytically in optimization problems; instead, high-fidelity simulations are employed to approximate the complex nonlinear relationships inherent in multiphase fluid flow within the network.
In other words, practical considerations deem the VLP curve a black-box function accessible only through expensive, single-point evaluations.

In this work, we employ the VLP curve to compute the \emph{upstream} pressure of a component as a function of its downstream pressure, the liquid flow rate, the mixture components (shares of water, oil, and gas), and the rate of lift-gas injected.
For a component $c$ of the system, we formulate the VLP curve as a function:
\begin{equation}
\begin{aligned}
    f_{\text{us}}^c: \R_+ \times \R_+ \times \R_+ \times \R_+ \times \left[ 0,1 \right]  &\longrightarrow \R_+ \\
    q_{\text{liq}}^{c}, p_{\text{ds}}^{c}, \text{iglr}^{c}, \text{gor}^{c}, \text{wct}^{c} &\longmapsto p_{\text{us}}^c
,\end{aligned}
\end{equation}
in which $q_{\text{liq}}^{c}$ is the liquid flow rate at that component, $p_{\text{ds}}^{c}$ is its downstream pressure, $\text{iglr}^{c}$ is the injected-gas to liquid ratio, and $\text{gor}^{c}$ and $\text{wct}^{c}$ are, respectively, the gas-oil ratio and the water-cut ratio of the flowing mixture.
Note that, for every component, $p_{\text{ds}}^{c}$ and $\text{iglr}^{c}$ are directly controlled.
The former is done by opening and closing choke valves, and the latter by lift-gas injection.

VLP curves are computed for each oil well in the system, taking as reference the reservoir pressure ($p_{\text{us}}^{c}$) and the pressure at the choke valve just upstream of the nearest connection ($p_{\text{ds}}^c$).
Therefore, for satellite wells, which are connected directly to the production platform, $p_{\text{ds}}^{c}$ is the pressure at the separator inlet to which the well is connected, whereas for manifold wells, which are connected to subsea manifolds, $p_{\text{ds}}^{c}$ is the pressure at its manifold valve.
VLP is also computed for the manifolds to account for the pressure drop along the riser pipeline, such that the upstream pressure is the pressure after the wells' valves, and the downstream pressure is the pressure at the valve connected to the separator.

\subsection{Network Modeling}

Let $\mathcal{W}$ be the set of all oil wells and $\mathcal{M}$, the set of all subsea manifolds.
  Further specializing, let $\mathcal{W}_{\text{sat}}$ be the set of all satellite wells (i.e., those that are connected directly to the production platform) and $\mathcal{W}_{\text{man}}= \bigcup_{m \in  \mathcal{M}} \mathcal{W}_{\text{man}}^{m}$ be the set of manifold wells, where $\mathcal{W}_{\text{man}}^{m}$ are the wells connected to manifold $m$.
\new{Then, the oil production optimization problem is formulated as}
\begin{subequations}\label{eq:oil-model}
\begin{align}
    \max \quad & \sum_{w \in \mathcal{W}} q_{\text{oil}}^{w} & \nonumber \\
	\textrm{s.t.} \quad & \left . \begin{aligned} 
        & \text{wct}^{c} = q_{\text{water}}^{c} / q_{\text{liq}}^{c} \\
	& q_{\text{oil}}^{c} = q_{\text{liq}}^{c} - q_{\text{water}}^{c} \\
	& \text{gor}^{c} = q_{\text{gas}}^{c} / q_{\text{oil}}^{c} \\
	& \text{iglr}^{c} = q_{\text{inj}}^{c} / q_{\text{liq}}^{c}
\end{aligned} \right \} & \forall c \in \mathcal{W}\cup \mathcal{M} \label{eq:oil-model-c1} \\
    \nonumber\\
  &  \left .\begin{aligned}
	& q_{\text{oil}}^{m} = \sum_{w \in \mathcal{W}^{m}_{\text{man}}} q_{\text{oil}}^{w} \\
        & q_{\text{water}}^{m} = \sum_{w \in \mathcal{W}^{m}_{\text{man}}} q_{\text{water}}^{w} \\
        & q_{\text{gas}}^{m} = \sum_{w \in \mathcal{W}^{m}_{\text{man}}} q_{\text{gas}}^{w} \\
        & q_{\text{inj}}^{m} = \sum_{w \in \mathcal{W}^{m}_{\text{man}}} q_{\text{inj}}^{w} \\
    \end{aligned} \right \} & \forall m \in \mathcal{M} \label{eq:oil-model-c2} \\
    \nonumber\\
    & p_{\text{us}}^{c} = f_{\text{us}}^{c}(q_{\text{liq}}^{c}, p_{\text{ds}}^{c}, \text{iglr}^{c}, \text{gor}^{c},\text{wct}^{c}) & \forall c \in \mathcal{W}\cup \mathcal{M} \label{eq:oil-model-c4} \\
    \nonumber\\
    & q_{\text{liq}}^{w} = \text{PI}^{w} \cdot  ( \overline{p}_{\text{res}}^{w} \cdot y^{w} - p_{\text{us}}^{w}) & \forall w \in \mathcal{W} \label{eq:oil-model-c5} \\
    \nonumber\\
   & y^m \le \sum_{w\in \mathcal{W}_{\text{man}}^{m}} y^w & \forall m\in \mathcal{M} \label{eq:oil-model-c12a} \\
   & y^w \leq  y^m & \forall w\in \mathcal{W}_{\text{man}}^{m} ,\,\forall m\in \mathcal{M} \label{eq:oil-model-c12b} \\
    \nonumber\\
	& \sum_{w\in \mathcal{W}} q_{\text{liq}}^{w} \le \overline{q}_{\text{liq}}^{\text{max}} & \label{eq:oil-model-c6a} \\
	& \sum_{w\in \mathcal{W}} q_{\text{inj}}^{w} \le \overline{q}_{\text{inj}}^{\text{max}} & \label{eq:oil-model-c6b} \\
    \nonumber\\
    & \Delta p_{\text{choke}}^{w} = p_{\text{ds}}^{w} - \overline{p}_{\text{sep}} & \forall w\in \mathcal{W}_{\text{sat}} \label{eq:oil-model-c7} \\
    & \Delta p_{\text{choke}}^{w} = p_{\text{ds}}^{w} - p_{\text{us}}^{m} & \forall w\in \mathcal{W}_{\text{man}}^{m},\,\forall m\in \mathcal{M} \label{eq:oil-model-c8} \\
    & \Delta p_{\text{choke}}^{m} = p_{\text{ds}}^{m} - \overline{p}_{\text{sep}} & \forall m\in \mathcal{M} \label{eq:oil-model-c9} \\
    \nonumber\\
  &  \left . \begin{aligned}
	& p_{\text{us}}^{c} - p_{\text{ds}}^{c} \ge -M \cdot (1-y^{c}) \\
	& y^{c} \in \left\{ 0,1 \right\} 
    \end{aligned} \right \} & \forall c \in \mathcal{W}\cup \mathcal{M} \label{eq:oil-model-c10}  \\
    \nonumber\\
    & \left . \begin{aligned}
    & t_{\text{GL}}^w\cdot \overline{q}_{\text{inj,min}}^w \le q_{\text{inj}}^{w} \le t_{\text{GL}}^w\cdot \overline{q}_{\text{inj,max}}^w  \\
    & t_{\text{GL}}^w \in \{0,1\} 
    \end{aligned} \right \} & \forall w \in \mathcal{W} \label{eq:oil-model-c11}
.\end{align}
\end{subequations}
where $\text{PI}^w$ is the productivity index (PI) of well $w$,  $\overline{p}_{\text{res}}^{w}$ is the reservoir pressure, $\overline{q}_{\text{liq}}^{\text{max}}$ and $\overline{q}_{\text{inj}}^{\text{max}}$ establish the platform's processing capacity, $\overline{p}_{\text{sep}}$ is the separator pressure, $\overline{q}_{\text{inj,min}}^w$ and $\overline{q}_{\text{inj,max}}^w$ gas-injection bounds on wells when operating, and $M$ is a sufficiently large constant.
These terms are parameters of the model, which are differentiated from variables by a monospaced typography.

The objective aims to maximize the total oil production, which promotes economic gain.
All variables are restricted to be non-negative, which is omitted for simplicity.
Effectively, the only decision variables are the lift-gas injection rates $q_{\text{inj}}^{w}$ and the pressure drops $\Delta p_{\text{choke}}^{c}$ across choke valves, as all the others are uniquely determined by assignments of those.
Exceptions are the $y^{c}$ binary variables, which determine whether component $c$ is active (open) or not (closed), and the binary variables $t_{\text{GL}}^w$ that impose limits on lift-gas injection when a well is operating. 

A brief description of each constraint follows:
\begin{itemize}
    \item \eqref{eq:oil-model-c1} ensures that the mixture fractions correspond to the flow rates in each component of the network. However, the relations are modeled with bilinear terms to avoid numerical problems involving the division operation, e.g., the relation $wct^w = q_{\text{water}}^w/q_{\text{liq}}^w$ is modeled as $\text{wct}^w\cdot q_{\text{liq}}^w = q_{\text{water}}^w$.
    
    \item \eqref{eq:oil-model-c2} models the mixture of the production streaming into a manifold from its connected wells without flow splitting, which is allowed in very peculiar conditions \citep{Camponogara:Flow-Split:2024}.
    
    \item \eqref{eq:oil-model-c4} defines the VLP curve for each component of the network.
    
    \item \eqref{eq:oil-model-c5} is the Inflow Performance Relationship (IPR), modeled here as a linear constraint, but could be treated as another black-box function for increased precision (see, for example, \cite{Vogel:JPT:1968}).\footnote{%
        Considering the approximation of the black-box function $p_{\text{us}}^{w}$ as a multilinear interpolation model, notice that $p_{\text{us}}^w=0$ when well $w$ is inactive because the convex combination of well $w$'s variables add up to $y^w$.  
        Consequently, when well $w$ is closed $\implies y^w=0 \implies p_{\text{us}}^w=q_{\text{liq}}^w=q_{\text{oil}}^w = q_{\text{water}}^w = q_{\text{gas}}^w=0$.
        Therefore, constraint \eqref{eq:oil-model-c5} is satisfied with $q_{\text{liq}}^{w}=0$.
    }

    \item \normalfont(\ref{eq:oil-model-c12a}-\ref{eq:oil-model-c12b}) ensure that the manifold is active if, and only if, at least one of the connected wells is active.
    
    \item \normalfont(\ref{eq:oil-model-c6a}-\ref{eq:oil-model-c6b}) implement the platform constraints limiting the liquid production (separator capacity) and the total lift-gas injected (compressor capacity).
    
    \item \normalfont(\ref{eq:oil-model-c7}-\ref{eq:oil-model-c9}) establish the pressure relations between the components through choke valves, ensuring pressure balance. 
    
    \item \eqref{eq:oil-model-c10} controls the activation ($y^c=1$) and deactivation ($y^{c}=0$) of each component, ensuring feasibility of the model when the choke valves are completely closed.
    
    \item \eqref{eq:oil-model-c11} ensures that, if well $w$ is operating with lift-gas injection ($t_{\text{GL}}=1$), the operational limits are respected.
\end{itemize}

Note that problem \eqref{eq:oil-model} can  be reformulated to fit the format of problem \eqref{eq:problem}, enabling the use of our relax-fix-and-exclude algorithm to compute a global optimum.
The black-box functions are the VLP curves $p_{\text{us}}^c(\cdot)$ that model the pressure drop through the components, \new{which we approximate using multilinear interpolation}.
Although \eqref{eq:oil-model} contains linear and bilinear constraints not explicitly present in \eqref{eq:problem}, the former are trivially handled by the algorithm, and the latter \new{can relaxed in the same way as the multilinear interpolants}, as discussed in Section~\ref{sec:milp-relaxation}.

\section{Computational Experiments}\label{sec:experiments}

For our computational experiments, we consider an oil production optimization problem on offshore production platforms based on real-world cases from Petróleo Brasileiro S.A. (Petrobras). We compare our proposed algorithm, Relax-Fix-and-Exclude, with the Gurobi optimization software~\cite{gurobi}, the BARON solver~\cite{baron2018}, and the SCIP optimization suite~\cite{SCIP}.

\subsection{Problem Instances}

\new{%
We consider nine scenarios (S1-S9) for the oil production optimization problem, each with a different arrangement of wells and manifolds.
For each scenario, we generate eight random instances by uniformly sampling the problem parameters from realistic intervals.
In all instances of all scenarios, the oil wells have the same geometry.
However, because the each well is drilled at different location of the reservoir, each has a unique VLP curve.
}

The VLP curves are computed through a simulation software, which produces the look-up tables.
We consider the GOR and WCT of each well as fixed parameters.
\new{Therefore, the resulting VLP curve for each well $w\in \mathcal{W}$ has breakpoints sampled from \emph{three} dimensions: $q_{\text{liq}}^{w}$, $p_{\text{ds}}^{w}$, and $\text{iglr}^{w}$.}
Each manifold also requires a VLP curve to model the pressure drop in the pipeline that connects the manifold to the oil platform (riser), as can be seen in Figure \ref{fig:esquema_plataforma}.
Because the oil-water-gas ratios of the mixture flowing out of the manifold is a function of the production from the connected wells, the VLP curve of the manifold must account for changes in the GOR and WCT inputs.
However, we consider that the choke valve downstream of each manifold is always completely open, such that $p_{\text{ds}}^m=p_{\text{sep}}$ for every manifold $m$.
\new{In other words, the VLP curve for each manifold $m\in \mathcal{M}$ has breakpoints sampled from \emph{four} dimensions: $q_{\text{liq}}^{m}$, $\text{iglr}^{m}$, $\text{gor}^{m}$, and $\text{wct}^{m}$.}

\new{%
The scenarios are ordered from S1 to S9 by complexity and expected difficulty of the instances, as can be seen in Table~\ref{tab:scenarios}.
The complexity increase comes from adding more components (oil wells and manifolds) to the model, which are equivalent to adding new multilinearly-interpolated black-box functions.
In particular, the addition of a manifold significantly increases the difficulty of the instance, even when compared to the addition of an oil well, as the manifold functions have an extra domain dimension.
}

\begin{table}[ht]
\centering
\caption{\new{Overview of the MINLP problem sizes in our experiments. For each scenario, eight instances were randomly generated by uniformly sampling the problem parameters. ``Nonzeros'' is the number of nonzero elements in the constraint matrix, indicating the problem sparsity.}} \label{tab:scenarios}
\label{tab:instances-overview}
\small
\new{%
\begin{tabular}{c  cc  cc  cc  c  c}
\toprule
Scenario 
  & \(|\mathcal{W}|\) 
  & \(|\mathcal{M}|\) 
  & \(|\bm{\xi}|\) 
  & \(|\bm{\lambda}|\) 
  & \(|\bm{y}|\) 
  & Rows 
  & Columns 
  & Nonzeros \\
\midrule
S1 & 1 & 0  &  50  &  4000 & 2  &  1889 &  4267 & 33505 \\
S2 & 2 & 0  & 100  &  8000 & 4  &  3777 &  8529 & 67011 \\
S3 & 3 & 0  & 150  & 12000 & 6  &  5690 & 12791 & 100542 \\
S4 & 4 & 1  & 255  & 46000 & 9  &  7549 & 47520 & 428042 \\
S5 & 5 & 1  & 305  & 50000 & 11 &  9391 & 51782 & 461510 \\
S6 & 6 & 1  & 355  & 54000 & 13 & 11149 & 56044 & 494886 \\
S7 & 7 & 2  & 460  & 88000 & 16 & 13296 & 90773 & 822666 \\
S8 & 8 & 2  & 510  & 92000 & 18 & 14857 & 95035 & 855861 \\
S9 & 9 & 2  & 560  & 96000 & 20 & 16681 & 99297 & 889303 \\
\bottomrule
\end{tabular}
}
\end{table}

\subsection{Relax-Fix-and-Exclude Implementation}

Our algorithm was implemented in Julia~\cite{Julia-2017}, using the JuMP library~\cite{Lubin2023} to manipulate the mathematical programming models.
We used Gurobi as the \texttt{NLPSolver} and the \texttt{MILPSolver}.
Note that Gurobi provides global optimality guarantees (up to specified tolerance) for non-convex quadratic problems, thus fulfilling our algorithm requirements.

We implement the exclusion cut through a custom SOS2 emulation using auxiliary binary variables, as discussed in Section~\ref{sec:excluding}.
\new{In particular, we emulate the SOS2 constraints using the \emph{lambda model}~\citep{dantzigSignificanceSolvingLinear1960}, also known as \emph{convex combination method}~\citep{vielma2010modeling}.}
This approach has the advantage of not being invasive in the solver's branching rules, although it results in larger MILP problems.

\subsection{Results and Discussion}

In our computational experiments, we measure the performance of the solution methods in terms of the best solution found within a 1-hour budget.
\new{We compare Relax-Fix-and-Exclude with three widely-used solvers for mixed-integer nonconvex problems: Gurobi~v11.0.3~\citep{gurobi}, SCIP~v9.0.0~\citep{SCIP}, and BARON~v24.5.8~\citep{baron2018}.
Recall that RFE requires the MILP relaxation to be formulated with the SOS2 constraints emulated through binary variables.
Furthermore, Gurobi and SCIP have their own implementation (native support) of SOS2 constraints, while BARON does not.
Thus, to make a fair evaluation, we consider Gurobi and SCIP with the MINLP problem formulated with both the SOS2 constraints and the emulation using binary variables.
In our experiments  ``Gurobi (SOS2)'' and ``SCIP (SOS2)'' indicate the use of the MINLP formulation with the solver's native implementation of the SOS2 constraints, while ``Gurobi (Bin.)'' and ``SCIP (Bin.)'' indicate that the formulation of the SOS2 with binary variables was used.}
All experiments reported below were performed on a workstation equipped with a 12th Gen Intel(R) Core(TM) i7-12700 processor (20 cores, 40 threads), 64 GB of RAM, and running Ubuntu 24.10 64-bit.

\new{%
The overall performance of the solution methods in terms of instances solved to optimality within the time limit is illustrated in Figure~\ref{fig:instances-time}.
We notice a significant performance gap between Gurobi (with either formulation) and RFE, in comparison to BARON and SCIP.
Furthermore, RFE was the solution method capable of solving most instances within the time limit.
}

\begin{figure}[ht]
    \centering
    \includegraphics{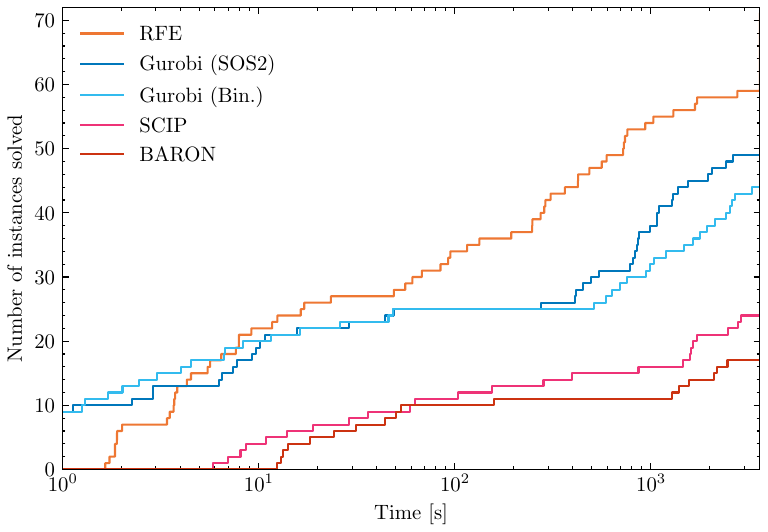}
    \caption{\new{Total number of instances solved to optimality by each solution method within the time limit.}}
    \label{fig:instances-time}
\end{figure}

\new{%
Table~\ref{tab:instances-solved} shows further details of the number of instances solved within each scenario.
We see that RFE and Gurobi were the only solution methods able to find feasible solutions to all instances in our experiments.
In contrast, BARON and SCIP were deeply impacted by the addition of a manifold in scenarios 4 to 9, indicating poor scalability with the input dimension of the multilinear interpolants, or, equivalently, with the number of variables in the nonlinear constraints.
Furthermore, RFE outperforms Gurobi in total number of instances solved to optimality mostly due to its better performance on the most difficult instances (scenarios 7 to 9).
}

\begin{table}[ht]
\centering
\caption{\new{Number of instances solved to optimality in each scenario of the oil production optimization problem, within 1 hour. The value in parenthesis indicates the number of instances for which the solution methods found at least one non-trivial feasible solution. All scenarios have the same number of instances (8).}} \label{tab:instances-solved}
\new{%
\begin{tabular}{l
                cccccccccc}
\toprule
Solver & S1 & S2 & S3 & S4 & S5 & S6 & S7 & S8 & S9 & Total \\
\midrule
RFE              & 8 (8) & 8 (8) & 8 (8) & 8 (8) & 8 (8) & 8 (8) & 6 (8) & 3 (8) & 2 (8) & 59 (72) \\
Gurobi (SOS2)    & 8 (8) & 8 (8) & 8 (8) & 8 (8) & 6 (8) & 8 (8) & 2 (8) & 1 (8) & 0 (8) & 49 (72) \\
Gurobi (Bin.)    & 8 (8) & 8 (8) & 8 (8) & 7 (8) & 6 (8) & 5 (8) & 1 (8) & 0 (8) & 1 (8) & 44 (72) \\
SCIP             & 8 (8) & 8 (8) & 8 (8) & 0 (1) & 0 (0) & 0 (0) & 0 (0) & 0 (0) & 0 (0) & 24 (25) \\
BARON            & 8 (8) & 7 (7) & 2 (3) & 0 (0) & 0 (0) & 0 (0) & 0 (0) & 0 (0) & 0 (0) & 17 (18) \\
\bottomrule
\end{tabular}}
\end{table}

\new{%
Suppose we restrict ourselves to the easy instances (without manifolds) from scenarios 1 to 3. In that case, we see mixed results, with Gurobi outperforming RFE most of the time, while SCIP and BARON are significantly worse than either.
Table~\ref{tab:times} summarizes these results in columns S1 to S3.
Notice that the easier the instance, the higher the margin that Gurobi has over RFE, particularly through the problem formulation that uses binary variables.
This explains the better performance of Gurobi indicated by Figure~\ref{fig:instances-time} for small time budgets (under 5 seconds).
However, the conclusion is reversed for the larger instances, as shown in Table~\ref{tab:times}, columns S4 to S6.
In fact, if we take into consideration all instances that both RFE and Gurobi (SOS2) solved to optimality in scenarios 3 to 9, RFE was on average 52\% faster than Gurobi (SOS2).
}

\begin{table}[ht]
\centering
\caption{\new{Average time to solve the instances to optimality in scenarios 1 to 6. The result is absent for SCIP and BARON in columns S4 to S6 because those solvers did not find optimal solutions to any instances in those scenarios. An asterisk (*) indicates that the solver did \emph{not} find an optimal solution for some instances of the respective scenario, and the total time budget was considered for the average in those cases.}} \label{tab:times}
\new{%
\begin{tabular}{l
                cccccccccc}
\toprule
Solver          & S1          & S2          & S3             & S4             & S5             & S6          \\
\midrule
RFE             &        3.13 &         3.85 & \textbf{7.86} & \textbf{353.30}& \textbf{335.38}& \textbf{313.21}\\
Gurobi (SOS2)   &        0.83 &         3.63 &        17.51  &         924.85 &        1571.51*&        835.97  \\
Gurobi (Bin.)   &\textbf{0.49}& \textbf{2.14}&        15.91  &        1769.34*&        2035.02*&       2030.69* \\
SCIP            &       21.91 &       241.43 &      2409.12  &             -- &           --   &            --  \\
BARON           &       37.38 &      1271.03*&      3271.66* &             -- &           --   &            --  \\
\bottomrule
\end{tabular}
}
\end{table}

\new{%
In scenarios 7 to 9, for which none of the methods solves all instances to optimality, we evaluate performance in terms of the average primal-dual gap.
Because SCIP and BARON did not find feasible solutions to the instances in those scenarios, they are excluded from the analysis.
The results summarized in Table~\ref{tab:gap} indicate that even when RFE does not find an optimal solution, the returned candidate solution is still of high quality.
In fact, in all instances of scenarios 7 to 9, RFE terminated with a gap below 1\%, significantly better than Gurobi with both formulations.
Furthermore, analyzing all instances in which neither RFE nor Gurobi proved optimality, RFE resulted in a gap \emph{at least} 35\% smaller, and on average 91\% smaller.
}

\begin{table}[ht]
\centering
\caption{\new{Percentual gap of the different solution methods in scenarios 7 to 9. The value shown is the average (maxium) over all instances.}} \label{tab:gap}
\new{%
\begin{tabular}{l
                ccccccc}
\toprule
Solver          & S7          & S8          & S9          \\
\midrule
RFE             & \textbf{0.01  (0.03)} & \textbf{0.12 (0.33)}  & \textbf{0.13 (0.57)}   \\
Gurobi (SOS2)   &         3.57 (24.65)  &        19.27 (49.31)  &         32.17 (109.65) \\
Gurobi (Bin.)   &         9.04 (30.61)  &        63.28 (229.00) &         85.75 (356.28) \\
\bottomrule
\end{tabular}
}
\end{table}

\new{%
We further illustrate this aspect of RFE through two instances from scenarios 7 and 9, for which the performance of RFE and Gurobi is illustrated in Figure~\ref{fig:results_example}.
In both examples, RFE takes longer than Gurobi to find a non-trivial feasible solution, but the solution it finds is already of very high quality (low gap).
Indeed, this behavior was observed in all of our experiments, that is, over all instances, the largest gap observed for the first feasible solution found by RFE was 1.2\%.
}

\begin{figure}[H]
    \centering
    \begin{subfigure}[t]{0.49\textwidth}
        \includegraphics[width=\textwidth]{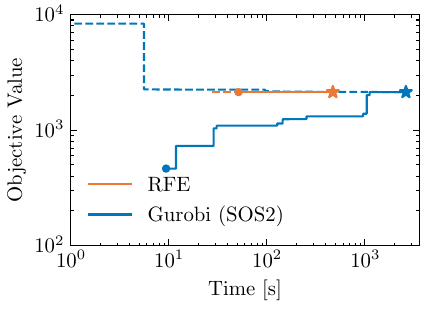}
        \caption{S7, Instance 1}
    \end{subfigure}
    \begin{subfigure}[t]{0.49\textwidth}
        \includegraphics[width=\textwidth]{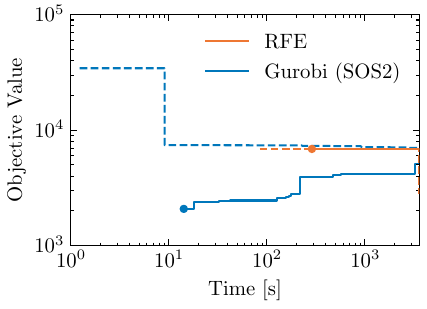}
        \caption{S9, Instance 8}
    \end{subfigure}
    \caption{Lower bound (objective value) progression during the time limit on two large instances of the oil production optimization problem. Dashed lines indicate upper (dual) bounds. Markers indicate the first feasible solution found and the optimal (if found).}
\label{fig:results_example} 
\end{figure}


\section{Conclusions} \label{sec:conclusion}

This paper introduced a novel algorithm for Mixed-Integer Nonlinear Programming (MINLP) problems involving multilinear interpolations of look-up tables.
The algorithm leverages a piecewise-linear relaxation that defines the convex envelope of the multilinear constraints, enabling efficient exploration of the solution space through variable fixing and exclusion strategies.
We have provided proof of convergence for the algorithm, demonstrating that it always returns the optimal solution or a valid statement of infeasibility or unboundedness.

Our computational experiments on oil production optimization problems demonstrated the effectiveness of the proposed method, which consistently outperformed all state-of-the-art solvers.
The results show that our algorithm, RFE, not only was able to provide optimal solutions faster, but also that it is a better heuristic approach (i.e., when used with a limited time budget) by providing better feasible solutions.
\new{On the instances with one manifold, RFE was able to find an optimum for all of them, including those that the alternative solvers could not solve to optimality within the time budget.  
Further, for the instances with two manifolds, the average gap of RFE was orders of magnitude smaller than the others.}
We highlight that these results corroborate our initial intuition that the proposed \new{MILP relaxation is of high quality and} results in a better enumeration of the combinatorial solution space.

While the results are promising, several avenues for future research could further enhance the performance and applicability of the algorithm.
First, the proposed relaxation can be applied to a broader range of nonlinearities, as it is not inherently tailored to multilinear constraints.
\new{Second, solving the MILP relaxations remains the primary computational bottleneck, consuming, on average, 85\% of the total runtime.}
To mitigate this, future implementations could explore methods for reusing solver information across iterations.
For instance, generating warm-starts using incumbent solutions or recycling parts of the branch-and-bound tree unaffected by exclusion cuts could significantly reduce computational overhead.

Additionally, we expect the runtime to be considerably reduced by a more efficient implementation of the exclusion cuts over the SOS2 constraints.
For example, implementing the cuts through modifications of the branching rules, \new{exploiting the classical approach of \citeauthor{bealeBranchBoundMethods1979}~\cite{bealeBranchBoundMethods1979}}, could result in a more efficient formulation of the MILP relaxation, as it would have a significantly reduced number of integer variables.

\new{%
Another approach would be to explore the formulations for SOS2 constraints.
The literature presents multiple alternatives~\citep{vielma2010modeling,Silva:EJOR:2014} that might suit combinatorial cuts better than the convex combination, used in our experiments.
Furthermore, the entire problem could admit different formulations based on the choice of representation of the SOS2 constraints using binary variables.
Even using the convex combination, as in our experiments, additional choice constraints could be used to significantly reduce the number of continuous variables in our MILP relaxation.
More precisely, one could formulate the problem using only $n$ continuous $\xi$ variables, and $2^n$ continuous $\lambda$ variables.
That would introduce products between those and the SOS2 binary variables, which could be handled through McCormick linearization in an extended formulation.
}

\backmatter

\bmhead{Acknowledgments}

This research was funded in part by Petr\'oleo Brasileiro S.A. (Petrobras) under grant SAP N$^{\text{o}}$ 4600677797, CAPES under grant 88887.827578/2023-00, and CNPq under grants 403092/2024-8, 402099/2023-0, and 308624/2021-1.

\bmhead{Code and Data Availability}

Our complete implementation of the Relax-Fix-and-Exclude algorithm as well as all necessary data to reproduce our results is available at \url{https://github.com/gos-ufsc/pwnl-oil}.

\begin{appendices}

\end{appendices}


\bibliography{references} 


\begin{thebibliography}{49}
\ifx \bisbn   \undefined \def \bisbn  #1{ISBN #1}\fi
\ifx \binits  \undefined \def \binits#1{#1}\fi
\ifx \bauthor  \undefined \def \bauthor#1{#1}\fi
\ifx \batitle  \undefined \def \batitle#1{#1}\fi
\ifx \bjtitle  \undefined \def \bjtitle#1{#1}\fi
\ifx \bvolume  \undefined \def \bvolume#1{\textbf{#1}}\fi
\ifx \byear  \undefined \def \byear#1{#1}\fi
\ifx \bissue  \undefined \def \bissue#1{#1}\fi
\ifx \bfpage  \undefined \def \bfpage#1{#1}\fi
\ifx \blpage  \undefined \def \blpage #1{#1}\fi
\ifx \burl  \undefined \def \burl#1{\textsf{#1}}\fi
\ifx \doiurl  \undefined \def \doiurl#1{\url{https://doi.org/#1}}\fi
\ifx \betal  \undefined \def \betal{\textit{et al.}}\fi
\ifx \binstitute  \undefined \def \binstitute#1{#1}\fi
\ifx \binstitutionaled  \undefined \def \binstitutionaled#1{#1}\fi
\ifx \bctitle  \undefined \def \bctitle#1{#1}\fi
\ifx \beditor  \undefined \def \beditor#1{#1}\fi
\ifx \bpublisher  \undefined \def \bpublisher#1{#1}\fi
\ifx \bbtitle  \undefined \def \bbtitle#1{#1}\fi
\ifx \bedition  \undefined \def \bedition#1{#1}\fi
\ifx \bseriesno  \undefined \def \bseriesno#1{#1}\fi
\ifx \blocation  \undefined \def \blocation#1{#1}\fi
\ifx \bsertitle  \undefined \def \bsertitle#1{#1}\fi
\ifx \bsnm \undefined \def \bsnm#1{#1}\fi
\ifx \bsuffix \undefined \def \bsuffix#1{#1}\fi
\ifx \bparticle \undefined \def \bparticle#1{#1}\fi
\ifx \barticle \undefined \def \barticle#1{#1}\fi
\bibcommenthead
\ifx \bconfdate \undefined \def \bconfdate #1{#1}\fi
\ifx \botherref \undefined \def \botherref #1{#1}\fi
\ifx \url \undefined \def \url#1{\textsf{#1}}\fi
\ifx \bchapter \undefined \def \bchapter#1{#1}\fi
\ifx \bbook \undefined \def \bbook#1{#1}\fi
\ifx \bcomment \undefined \def \bcomment#1{#1}\fi
\ifx \oauthor \undefined \def \oauthor#1{#1}\fi
\ifx \citeauthoryear \undefined \def \citeauthoryear#1{#1}\fi
\ifx \endbibitem  \undefined \def \endbibitem {}\fi
\ifx \bconflocation  \undefined \def \bconflocation#1{#1}\fi
\ifx \arxivurl  \undefined \def \arxivurl#1{\textsf{#1}}\fi
\csname PreBibitemsHook\endcsname

\bibitem[\protect\citeauthoryear{Jones et~al.}{1998}]{jonesEfficientGlobalOptimization1998}
\begin{barticle}
\bauthor{\bsnm{Jones}, \binits{D.R.}},
\bauthor{\bsnm{Schonlau}, \binits{M.}},
\bauthor{\bsnm{Welch}, \binits{W.J.}}:
\batitle{Efficient global optimization of expensive black-box functions}.
\bjtitle{Journal of Global Optimization}
\bvolume{13}(\bissue{4}),
\bfpage{455}--\blpage{492}
(\byear{1998})
\doiurl{10.1023/A:1008306431147}
\end{barticle}
\endbibitem

\bibitem[\protect\citeauthoryear{Bischl et~al.}{2023}]{bischlHyperparameterOptimizationFoundations2023}
\begin{barticle}
\bauthor{\bsnm{Bischl}, \binits{B.}},
\bauthor{\bsnm{Binder}, \binits{M.}},
\bauthor{\bsnm{Lang}, \binits{M.}},
\bauthor{\bsnm{Pielok}, \binits{T.}},
\bauthor{\bsnm{Richter}, \binits{J.}},
\bauthor{\bsnm{Coors}, \binits{S.}},
\bauthor{\bsnm{Thomas}, \binits{J.}},
\bauthor{\bsnm{Ullmann}, \binits{T.}},
\bauthor{\bsnm{Becker}, \binits{M.}},
\bauthor{\bsnm{Boulesteix}, \binits{A.-L.}},
\bauthor{\bsnm{Deng}, \binits{D.}},
\bauthor{\bsnm{Lindauer}, \binits{M.}}:
\batitle{Hyperparameter optimization: {{Foundations}}, algorithms, best practices, and open challenges}.
\bjtitle{WIREs Data Mining and Knowledge Discovery}
\bvolume{13}(\bissue{2}),
\bfpage{1484}
(\byear{2023})
\doiurl{10.1002/widm.1484}
\end{barticle}
\endbibitem

\bibitem[\protect\citeauthoryear{Shields et~al.}{2021}]{shieldsBayesianReactionOptimization2021}
\begin{barticle}
\bauthor{\bsnm{Shields}, \binits{B.J.}},
\bauthor{\bsnm{Stevens}, \binits{J.}},
\bauthor{\bsnm{Li}, \binits{J.}},
\bauthor{\bsnm{Parasram}, \binits{M.}},
\bauthor{\bsnm{Damani}, \binits{F.}},
\bauthor{\bsnm{Alvarado}, \binits{J.I.M.}},
\bauthor{\bsnm{Janey}, \binits{J.M.}},
\bauthor{\bsnm{Adams}, \binits{R.P.}},
\bauthor{\bsnm{Doyle}, \binits{A.G.}}:
\batitle{Bayesian reaction optimization as a tool for chemical synthesis}.
\bjtitle{Nature}
\bvolume{590}(\bissue{7844}),
\bfpage{89}--\blpage{96}
(\byear{2021})
\doiurl{10.1038/s41586-021-03213-y}
\end{barticle}
\endbibitem

\bibitem[\protect\citeauthoryear{Rao}{2019}]{raoEngineeringOptimizationTheory2019a}
\begin{bbook}
\bauthor{\bsnm{Rao}, \binits{S.S.}}:
\bbtitle{Engineering Optimization: Theory and Practice},
\bedition{1}st edn.
\bpublisher{Wiley},
\blocation{Hoboken, NJ}
(\byear{2019}).
\doiurl{10.1002/9781119454816}
\end{bbook}
\endbibitem

\bibitem[\protect\citeauthoryear{Epelle and Gerogiorgis}{2020}]{epelleComputationalPerformanceComparison2020a}
\begin{barticle}
\bauthor{\bsnm{Epelle}, \binits{E.I.}},
\bauthor{\bsnm{Gerogiorgis}, \binits{D.I.}}:
\batitle{A computational performance comparison of {{MILP}} vs. {{MINLP}} formulations for oil production optimisation}.
\bjtitle{Computers \& Chemical Engineering}
\bvolume{140},
\bfpage{106903}
(\byear{2020})
\doiurl{10.1016/j.compchemeng.2020.106903}
\end{barticle}
\endbibitem

\bibitem[\protect\citeauthoryear{Grimstad et~al.}{2016}]{grimstadGlobalOptimizationMultiphase2016}
\begin{barticle}
\bauthor{\bsnm{Grimstad}, \binits{B.}},
\bauthor{\bsnm{Foss}, \binits{B.}},
\bauthor{\bsnm{Heddle}, \binits{R.}},
\bauthor{\bsnm{Woodman}, \binits{M.}}:
\batitle{Global optimization of multiphase flow networks using spline surrogate models}.
\bjtitle{Computers \& Chemical Engineering}
\bvolume{84},
\bfpage{237}--\blpage{254}
(\byear{2016})
\doiurl{10.1016/j.compchemeng.2015.08.022}
\end{barticle}
\endbibitem

\bibitem[\protect\citeauthoryear{{Gastellu-Etchegorry} et~al.}{2003}]{gastellu-etchegorryInterpolationProcedureGeneralizing2003}
\begin{barticle}
\bauthor{\bsnm{{Gastellu-Etchegorry}}, \binits{J.P.}},
\bauthor{\bsnm{Gascon}, \binits{F.}},
\bauthor{\bsnm{Est{\`e}ve}, \binits{P.}}:
\batitle{An interpolation procedure for generalizing a look-up table inversion method}.
\bjtitle{Remote Sensing of Environment}
\bvolume{87}(\bissue{1}),
\bfpage{55}--\blpage{71}
(\byear{2003})
\doiurl{10.1016/S0034-4257(03)00146-9}
\end{barticle}
\endbibitem

\bibitem[\protect\citeauthoryear{Pini et~al.}{2015}]{piniConsistentLookupTable2015}
\begin{barticle}
\bauthor{\bsnm{Pini}, \binits{M.}},
\bauthor{\bsnm{Spinelli}, \binits{A.}},
\bauthor{\bsnm{Persico}, \binits{G.}},
\bauthor{\bsnm{Rebay}, \binits{S.}}:
\batitle{Consistent look-up table interpolation method for real-gas flow simulations}.
\bjtitle{Computers \& Fluids}
\bvolume{107},
\bfpage{178}--\blpage{188}
(\byear{2015})
\doiurl{10.1016/j.compfluid.2014.11.001}
\end{barticle}
\endbibitem

\bibitem[\protect\citeauthoryear{Nelles and Fink}{2000}]{nellesGridBasedLookUpTable2000}
\begin{barticle}
\bauthor{\bsnm{Nelles}, \binits{O.}},
\bauthor{\bsnm{Fink}, \binits{A.}}:
\batitle{Grid-{{Based Look-Up Table Optimization Toolbox}}}.
\bjtitle{IFAC Proceedings Volumes}
\bvolume{33}(\bissue{15}),
\bfpage{839}--\blpage{844}
(\byear{2000})
\doiurl{10.1016/S1474-6670(17)39857-9}
\end{barticle}
\endbibitem

\bibitem[\protect\citeauthoryear{Nelles}{2020}]{nellesLinearPolynomialLookUp2020}
\begin{bchapter}
\bauthor{\bsnm{Nelles}, \binits{O.}}:
\bctitle{Linear, polynomial, and look-up table models}.
In: \beditor{\bsnm{Nelles}, \binits{O.}} (ed.)
\bbtitle{Nonlinear System Identification: From Classical Approaches to Neural Networks, Fuzzy Models, and Gaussian Processes},
pp. \bfpage{249}--\blpage{278}.
\bpublisher{Springer},
\blocation{Cham}
(\byear{2020}).
\doiurl{10.1007/978-3-030-47439-3\_10}
\end{bchapter}
\endbibitem

\bibitem[\protect\citeauthoryear{Bohn et~al.}{2006}]{bohnOptimizationbasedApproachCalibration2006a}
\begin{bchapter}
\bauthor{\bsnm{Bohn}, \binits{C.}},
\bauthor{\bsnm{Stober}, \binits{P.}},
\bauthor{\bsnm{Magnor}, \binits{O.}}:
\bctitle{An optimization-based approach for the calibration of lookup tables in electronic engine control}.
In: \bbtitle{IEEE Conference on Computer Aided Control System Design, IEEE International Conference on Control Applications, IEEE International Symposium on Intelligent Control},
pp. \bfpage{2315}--\blpage{2320}
(\byear{2006}).
\doiurl{10.1109/CACSD-CCA-ISIC.2006.4777001}
\end{bchapter}
\endbibitem

\bibitem[\protect\citeauthoryear{Gupta et~al.}{2016}]{guptaMonotonicCalibratedInterpolated2016}
\begin{barticle}
\bauthor{\bsnm{Gupta}, \binits{M.}},
\bauthor{\bsnm{Cotter}, \binits{A.}},
\bauthor{\bsnm{Pfeifer}, \binits{J.}},
\bauthor{\bsnm{Voevodski}, \binits{K.}},
\bauthor{\bsnm{Canini}, \binits{K.}},
\bauthor{\bsnm{Mangylov}, \binits{A.}},
\bauthor{\bsnm{Moczydlowski}, \binits{W.}},
\bauthor{\bsnm{Esbroeck}, \binits{A.}}:
\batitle{Monotonic calibrated interpolated look-up tables}.
\bjtitle{Journal of Machine Learning Research}
\bvolume{17}(\bissue{109}),
\bfpage{1}--\blpage{47}
(\byear{2016})
\doiurl{10.48550/arXiv.1505.06378}
\end{barticle}
\endbibitem

\bibitem[\protect\citeauthoryear{Furlan et~al.}{2016}]{furlanSimpleApproachImprove2016}
\begin{barticle}
\bauthor{\bsnm{Furlan}, \binits{F.F.}},
\bauthor{\bsnm{{de Andrade Lino}}, \binits{A.R.}},
\bauthor{\bsnm{Matugi}, \binits{K.}},
\bauthor{\bsnm{Cruz}, \binits{A.J.G.}},
\bauthor{\bsnm{Secchi}, \binits{A.R.}},
\bauthor{\bsnm{{de Campos Giordano}}, \binits{R.}}:
\batitle{A simple approach to improve the robustness of equation-oriented simulators: {{Multilinear}} look-up table interpolators}.
\bjtitle{Computers \& Chemical Engineering}
\bvolume{86},
\bfpage{1}--\blpage{4}
(\byear{2016})
\doiurl{10.1016/j.compchemeng.2015.12.014}
\end{barticle}
\endbibitem

\bibitem[\protect\citeauthoryear{Martins et~al.}{2020}]{martinsImplementationMultilinearLookup2020}
\begin{barticle}
\bauthor{\bsnm{Martins}, \binits{C.d.O.}},
\bauthor{\bsnm{Furlan}, \binits{F.F.}},
\bauthor{\bsnm{Giordano}, \binits{R.d.C.}}:
\batitle{Implementation of multilinear look-up tables as surrogate models for the ethyl transesterification reactor in equation-oriented simulator {{EMSO}}}.
\bjtitle{The Journal of Engineering and Exact Sciences}
\bvolume{6}(\bissue{4}),
\bfpage{0467}--\blpage{0473}
(\byear{2020})
\doiurl{10.18540/jcecvl6iss4pp0467-0473}
\end{barticle}
\endbibitem

\bibitem[\protect\citeauthoryear{Beale}{1983}]{beale_mathematical_1983}
\begin{barticle}
\bauthor{\bsnm{Beale}, \binits{E.M.L.}}:
\batitle{A mathematical programming model for the long term development of an off-shore gas field}.
\bjtitle{Discrete Applied Mathematics}
\bvolume{5}(\bissue{1}),
\bfpage{1}--\blpage{9}
(\byear{1983})
\doiurl{10.1016/0166-218X(83)90011-2}
\end{barticle}
\endbibitem

\bibitem[\protect\citeauthoryear{Vielma et~al.}{2010}]{vielma2010modeling}
\begin{barticle}
\bauthor{\bsnm{Vielma}, \binits{J.P.}},
\bauthor{\bsnm{Ahmed}, \binits{S.}},
\bauthor{\bsnm{Nemhauser}, \binits{G.L.}}:
\batitle{Mixed-integer models for nonseparable piecewise linear optimization: Unifying framework and extensions}.
\bjtitle{Operations Research}
\bvolume{58}(\bissue{2}),
\bfpage{303}--\blpage{315}
(\byear{2010})
\doiurl{10.1287/opre.1090.0721}
\end{barticle}
\endbibitem

\bibitem[\protect\citeauthoryear{Silva and Camponogara}{2014}]{Silva:EJOR:2014}
\begin{barticle}
\bauthor{\bsnm{Silva}, \binits{T.L.}},
\bauthor{\bsnm{Camponogara}, \binits{E.}}:
\batitle{A computational analysis of multidimensional piecewise-linear models with applications to oil production optimization}.
\bjtitle{European Journal of Operational Research}
\bvolume{232}(\bissue{3}),
\bfpage{630}--\blpage{642}
(\byear{2014})
\doiurl{10.1016/j.ejor.2013.07.040}
\end{barticle}
\endbibitem

\bibitem[\protect\citeauthoryear{Huchette and Vielma}{2022}]{huchette2022piecewise}
\begin{barticle}
\bauthor{\bsnm{Huchette}, \binits{J.}},
\bauthor{\bsnm{Vielma}, \binits{J.P.}}:
\batitle{{Nonconvex piecewise linear functions: Advanced formulations and simple modeling tools}}.
\bjtitle{Operations Research}
\bvolume{71}(\bissue{5}),
\bfpage{1835}--\blpage{1856}
(\byear{2022})
\doiurl{10.1287/opre.2019.1973}
\end{barticle}
\endbibitem

\bibitem[\protect\citeauthoryear{Misener and Floudas}{2010}]{misenerPiecewiseLinearApproximationsMultidimensional2010}
\begin{barticle}
\bauthor{\bsnm{Misener}, \binits{R.}},
\bauthor{\bsnm{Floudas}, \binits{C.A.}}:
\batitle{Piecewise-linear approximations of multidimensional functions}.
\bjtitle{Journal of Optimization Theory and Applications}
\bvolume{145}(\bissue{1}),
\bfpage{120}--\blpage{147}
(\byear{2010})
\doiurl{10.1007/s10957-009-9626-0}
\end{barticle}
\endbibitem

\bibitem[\protect\citeauthoryear{Rikun}{1997}]{rikunConvexEnvelopeFormula1997}
\begin{barticle}
\bauthor{\bsnm{Rikun}, \binits{A.D.}}:
\batitle{A convex envelope formula for multilinear functions}.
\bjtitle{Journal of Global Optimization}
\bvolume{10}(\bissue{4}),
\bfpage{425}--\blpage{437}
(\byear{1997})
\doiurl{10.1023/A:1008217604285}
\end{barticle}
\endbibitem

\bibitem[\protect\citeauthoryear{Keha et~al.}{2004}]{kehaModelsRepresentingPiecewise2004}
\begin{barticle}
\bauthor{\bsnm{Keha}, \binits{A.B.}},
\bauthor{\bsnm{{de Farias}}, \binits{I.R.}},
\bauthor{\bsnm{Nemhauser}, \binits{G.L.}}:
\batitle{Models for representing piecewise linear cost functions}.
\bjtitle{Operations Research Letters}
\bvolume{32}(\bissue{1}),
\bfpage{44}--\blpage{48}
(\byear{2004})
\doiurl{10.1016/S0167-6377(03)00059-2}
\end{barticle}
\endbibitem

\bibitem[\protect\citeauthoryear{McCormick}{1976}]{McCormick1976}
\begin{barticle}
\bauthor{\bsnm{McCormick}, \binits{G.P.}}:
\batitle{Computability of global solutions to factorable nonconvex programs: Part {I} -- convex underestimating problems}.
\bjtitle{Mathematical Programming}
\bvolume{10}(\bissue{1}),
\bfpage{147}--\blpage{175}
(\byear{1976})
\doiurl{10.1007/bf01580665}
\end{barticle}
\endbibitem

\bibitem[\protect\citeauthoryear{Sundar et~al.}{2021}]{sundarPiecewisePolyhedralFormulations2021}
\begin{barticle}
\bauthor{\bsnm{Sundar}, \binits{K.}},
\bauthor{\bsnm{Nagarajan}, \binits{H.}},
\bauthor{\bsnm{Linderoth}, \binits{J.}},
\bauthor{\bsnm{Wang}, \binits{S.}},
\bauthor{\bsnm{Bent}, \binits{R.}}:
\batitle{Piecewise polyhedral formulations for a multilinear term}.
\bjtitle{Operations Research Letters}
\bvolume{49}(\bissue{1}),
\bfpage{144}--\blpage{149}
(\byear{2021})
\doiurl{10.1016/j.orl.2020.12.002}
\end{barticle}
\endbibitem

\bibitem[\protect\citeauthoryear{Kim et~al.}{2024}]{kimPiecewisePolyhedralRelaxations2024}
\begin{botherref}
\oauthor{\bsnm{Kim}, \binits{J.}},
\oauthor{\bsnm{Richard}, \binits{J.-P.P.}},
\oauthor{\bsnm{Tawarmalani}, \binits{M.}}:
Piecewise polyhedral relaxations of multilinear optimization.
SIAM Journal on Optimization,
3167--3193
(2024)
\doiurl{10.1137/22M1507486}
\end{botherref}
\endbibitem

\bibitem[\protect\citeauthoryear{{Gurobi Optimization, LLC}}{2024}]{gurobi}
\begin{botherref}
\oauthor{\bsnm{{Gurobi Optimization, LLC}}}:
{Gurobi Optimizer Reference Manual}
(2024).
\url{https://www.gurobi.com}
\end{botherref}
\endbibitem

\bibitem[\protect\citeauthoryear{Gleixner}{2015}]{gleixnerExactFastAlgorithms2015}
\begin{bbook}
\bauthor{\bsnm{Gleixner}, \binits{A.M.}}:
\bbtitle{Exact and Fast Algorithms for Mixed-Integer Nonlinear Programming}.
\bpublisher{Logos Verlag Berlin},
\blocation{Berlin, Germany}
(\byear{2015})
\end{bbook}
\endbibitem

\bibitem[\protect\citeauthoryear{Weiser and Zarantonello}{1988}]{weiser_note_1988}
\begin{barticle}
\bauthor{\bsnm{Weiser}, \binits{A.}},
\bauthor{\bsnm{Zarantonello}, \binits{S.E.}}:
\batitle{A note on piecewise linear and multilinear table interpolation in many dimensions}.
\bjtitle{Mathematics of Computation}
\bvolume{50}(\bissue{181}),
\bfpage{189}--\blpage{196}
(\byear{1988})
\doiurl{10.1090/S0025-5718-1988-0917826-0}
\end{barticle}
\endbibitem

\bibitem[\protect\citeauthoryear{Sahinidis}{1996}]{sahinidisBARONGeneralPurpose1996}
\begin{barticle}
\bauthor{\bsnm{Sahinidis}, \binits{N.V.}}:
\batitle{{{BARON}}: {{A}} general purpose global optimization software package}.
\bjtitle{Journal of Global Optimization}
\bvolume{8}(\bissue{2}),
\bfpage{201}--\blpage{205}
(\byear{1996})
\doiurl{10.1007/BF00138693}
\end{barticle}
\endbibitem

\bibitem[\protect\citeauthoryear{Zhang and Sahinidis}{2024}]{zhangSolvingContinuousDiscrete2024}
\begin{barticle}
\bauthor{\bsnm{Zhang}, \binits{Y.}},
\bauthor{\bsnm{Sahinidis}, \binits{N.V.}}:
\batitle{Solving continuous and discrete nonlinear programs with {{BARON}}}.
\bjtitle{Computational Optimization and Applications}
(\byear{2024})
\doiurl{10.1007/s10589-024-00633-0}
\end{barticle}
\endbibitem

\bibitem[\protect\citeauthoryear{Bolusani et~al.}{2024a}]{BolusaniEtal2024OO}
\begin{botherref}
\oauthor{\bsnm{Bolusani}, \binits{S.}},
\oauthor{\bsnm{Besan{\c{c}}on}, \binits{M.}},
\oauthor{\bsnm{Bestuzheva}, \binits{K.}},
\oauthor{\bsnm{Chmiela}, \binits{A.}},
\oauthor{\bsnm{Dion{\'{i}}sio}, \binits{J.}},
\oauthor{\bsnm{Donkiewicz}, \binits{T.}},
\oauthor{\bsnm{Doornmalen}, \binits{J.}},
\oauthor{\bsnm{Eifler}, \binits{L.}},
\oauthor{\bsnm{Ghannam}, \binits{M.}},
\oauthor{\bsnm{Gleixner}, \binits{A.}},
\oauthor{\bsnm{Graczyk}, \binits{C.}},
\oauthor{\bsnm{Halbig}, \binits{K.}},
\oauthor{\bsnm{Hedtke}, \binits{I.}},
\oauthor{\bsnm{Hoen}, \binits{A.}},
\oauthor{\bsnm{Hojny}, \binits{C.}},
\oauthor{\bsnm{Hulst}, \binits{R.}},
\oauthor{\bsnm{Kamp}, \binits{D.}},
\oauthor{\bsnm{Koch}, \binits{T.}},
\oauthor{\bsnm{Kofler}, \binits{K.}},
\oauthor{\bsnm{Lentz}, \binits{J.}},
\oauthor{\bsnm{Manns}, \binits{J.}},
\oauthor{\bsnm{Mexi}, \binits{G.}},
\oauthor{\bsnm{M\"{u}hmer}, \binits{E.}},
\oauthor{\bsnm{Pfetsch}, \binits{M.E.}},
\oauthor{\bsnm{Schl{\"o}sser}, \binits{F.}},
\oauthor{\bsnm{Serrano}, \binits{F.}},
\oauthor{\bsnm{Shinano}, \binits{Y.}},
\oauthor{\bsnm{Turner}, \binits{M.}},
\oauthor{\bsnm{Vigerske}, \binits{S.}},
\oauthor{\bsnm{Weninger}, \binits{D.}},
\oauthor{\bsnm{Xu}, \binits{L.}}:
{The SCIP Optimization Suite 9.0}.
Technical report,
Optimization Online
(February 2024).
\url{https://optimization-online.org/2024/02/the-scip-optimization-suite-9-0/}
\end{botherref}
\endbibitem

\bibitem[\protect\citeauthoryear{Bolusani et~al.}{2024b}]{BolusaniEtal2024ZR}
\begin{botherref}
\oauthor{\bsnm{Bolusani}, \binits{S.}},
\oauthor{\bsnm{Besan{\c{c}}on}, \binits{M.}},
\oauthor{\bsnm{Bestuzheva}, \binits{K.}},
\oauthor{\bsnm{Chmiela}, \binits{A.}},
\oauthor{\bsnm{Dion{\'{i}}sio}, \binits{J.}},
\oauthor{\bsnm{Donkiewicz}, \binits{T.}},
\oauthor{\bsnm{Doornmalen}, \binits{J.}},
\oauthor{\bsnm{Eifler}, \binits{L.}},
\oauthor{\bsnm{Ghannam}, \binits{M.}},
\oauthor{\bsnm{Gleixner}, \binits{A.}},
\oauthor{\bsnm{Graczyk}, \binits{C.}},
\oauthor{\bsnm{Halbig}, \binits{K.}},
\oauthor{\bsnm{Hedtke}, \binits{I.}},
\oauthor{\bsnm{Hoen}, \binits{A.}},
\oauthor{\bsnm{Hojny}, \binits{C.}},
\oauthor{\bsnm{Hulst}, \binits{R.}},
\oauthor{\bsnm{Kamp}, \binits{D.}},
\oauthor{\bsnm{Koch}, \binits{T.}},
\oauthor{\bsnm{Kofler}, \binits{K.}},
\oauthor{\bsnm{Lentz}, \binits{J.}},
\oauthor{\bsnm{Manns}, \binits{J.}},
\oauthor{\bsnm{Mexi}, \binits{G.}},
\oauthor{\bsnm{M\"{u}hmer}, \binits{E.}},
\oauthor{\bsnm{Pfetsch}, \binits{M.E.}},
\oauthor{\bsnm{Schl{\"o}sser}, \binits{F.}},
\oauthor{\bsnm{Serrano}, \binits{F.}},
\oauthor{\bsnm{Shinano}, \binits{Y.}},
\oauthor{\bsnm{Turner}, \binits{M.}},
\oauthor{\bsnm{Vigerske}, \binits{S.}},
\oauthor{\bsnm{Weninger}, \binits{D.}},
\oauthor{\bsnm{Xu}, \binits{L.}}:
{The SCIP Optimization Suite 9.0}.
ZIB-Report 24-02-29,
Zuse Institute Berlin
(February 2024).
\url{https://nbn-resolving.org/urn:nbn:de:0297-zib-95528}
\end{botherref}
\endbibitem

\bibitem[\protect\citeauthoryear{Biegler and Grossmann}{2004}]{BIEGLER20041169}
\begin{barticle}
\bauthor{\bsnm{Biegler}, \binits{L.T.}},
\bauthor{\bsnm{Grossmann}, \binits{I.E.}}:
\batitle{Retrospective on optimization}.
\bjtitle{Computers \& Chemical Engineering}
\bvolume{28}(\bissue{8}),
\bfpage{1169}--\blpage{1192}
(\byear{2004})
\doiurl{10.1016/j.compchemeng.2003.11.003}
\end{barticle}
\endbibitem

\bibitem[\protect\citeauthoryear{Beale and Tomlin}{1970}]{bealeTomlin1970}
\begin{bchapter}
\bauthor{\bsnm{Beale}, \binits{E.}},
\bauthor{\bsnm{Tomlin}, \binits{J.}}:
\bctitle{Special facilities in a general mathematical programming system for nonconvex problems using ordered sets of variables}.
In: \bbtitle{Proceedings of the Fifth International Conference on Operational Research},
pp. \bfpage{447}--\blpage{454}
(\byear{1970})
\end{bchapter}
\endbibitem

\bibitem[\protect\citeauthoryear{Castro}{2015}]{CASTRO2015300}
\begin{barticle}
\bauthor{\bsnm{Castro}, \binits{P.M.}}:
\batitle{Tightening piecewise {McCormick} relaxations for bilinear problems}.
\bjtitle{Computers \& Chemical Engineering}
\bvolume{72},
\bfpage{300}--\blpage{311}
(\byear{2015})
\doiurl{10.1016/j.compchemeng.2014.03.025}
\end{barticle}
\endbibitem

\bibitem[\protect\citeauthoryear{{Al-Khayyal} and Falk}{1983}]{al-khayyalJointlyConstrainedBiconvex1983}
\begin{barticle}
\bauthor{\bsnm{{Al-Khayyal}}, \binits{F.A.}},
\bauthor{\bsnm{Falk}, \binits{J.E.}}:
\batitle{Jointly constrained biconvex programming}.
\bjtitle{Mathematics of Operations Research}
\bvolume{8}(\bissue{2}),
\bfpage{273}--\blpage{286}
(\byear{1983})
\doiurl{10.1287/moor.8.2.273}
\end{barticle}
\endbibitem

\bibitem[\protect\citeauthoryear{D’Ambrosio et~al.}{2010}]{DAMBROSIO2010341}
\begin{barticle}
\bauthor{\bsnm{D’Ambrosio}, \binits{C.}},
\bauthor{\bsnm{Frangioni}, \binits{A.}},
\bauthor{\bsnm{Liberti}, \binits{L.}},
\bauthor{\bsnm{Lodi}, \binits{A.}}:
\batitle{On interval-subgradient and no-good cuts}.
\bjtitle{Operations Research Letters}
\bvolume{38}(\bissue{5}),
\bfpage{341}--\blpage{345}
(\byear{2010})
\doiurl{10.1016/j.orl.2010.05.010}
\end{barticle}
\endbibitem

\bibitem[\protect\citeauthoryear{Balas and Jeroslow}{1972}]{balas_jeroslaw_72}
\begin{barticle}
\bauthor{\bsnm{Balas}, \binits{E.}},
\bauthor{\bsnm{Jeroslow}, \binits{R.}}:
\batitle{Canonical cuts on the unit hypercube}.
\bjtitle{SIAM Journal on Applied Mathematics}
\bvolume{23}(\bissue{1}),
\bfpage{61}--\blpage{69}
(\byear{1972})
\doiurl{10.1137/0123007}
\end{barticle}
\endbibitem

\bibitem[\protect\citeauthoryear{Puranik and Sahinidis}{2017}]{puranikDomainReductionTechniques2017a}
\begin{barticle}
\bauthor{\bsnm{Puranik}, \binits{Y.}},
\bauthor{\bsnm{Sahinidis}, \binits{N.V.}}:
\batitle{Domain reduction techniques for global {{NLP}} and {{MINLP}} optimization}.
\bjtitle{Constraints}
\bvolume{22}(\bissue{3}),
\bfpage{338}--\blpage{376}
(\byear{2017})
\doiurl{10.1007/s10601-016-9267-5}
\end{barticle}
\endbibitem

\bibitem[\protect\citeauthoryear{Foss and Jenson}{2010}]{foss_jenson}
\begin{barticle}
\bauthor{\bsnm{Foss}, \binits{B.}},
\bauthor{\bsnm{Jenson}, \binits{J.P.}}:
\batitle{Performance analysis for closed-loop reservoir management}.
\bjtitle{SPE Journal}
\bvolume{16}(\bissue{01}),
\bfpage{183}--\blpage{190}
(\byear{2010})
\doiurl{10.2118/138891-PA}
\end{barticle}
\endbibitem

\bibitem[\protect\citeauthoryear{M{\"u}ller et~al.}{2022}]{Muller2022}
\begin{barticle}
\bauthor{\bsnm{M{\"u}ller}, \binits{E.R.}},
\bauthor{\bsnm{Camponogara}, \binits{E.}},
\bauthor{\bsnm{Seman}, \binits{L.O.}},
\bauthor{\bsnm{H{\"u}lse}, \binits{E.O.}},
\bauthor{\bsnm{Vieira}, \binits{B.F.}},
\bauthor{\bsnm{Miyatake}, \binits{L.K.}},
\bauthor{\bsnm{Teixeira}, \binits{A.F.}}:
\batitle{Short-term steady-state production optimization of offshore oil platforms: wells with dual completion (gas-lift and {ESP}) and flow assurance}.
\bjtitle{TOP}
\bvolume{30}(\bissue{1}),
\bfpage{152}--\blpage{180}
(\byear{2022})
\doiurl{10.1007/s11750-021-00604-2}
\end{barticle}
\endbibitem

\bibitem[\protect\citeauthoryear{Codas and Camponogara}{2012}]{CODAS2012222}
\begin{barticle}
\bauthor{\bsnm{Codas}, \binits{A.}},
\bauthor{\bsnm{Camponogara}, \binits{E.}}:
\batitle{Mixed-integer linear optimization for optimal lift-gas allocation with well-separator routing}.
\bjtitle{European Journal of Operational Research}
\bvolume{217}(\bissue{1}),
\bfpage{222}--\blpage{231}
(\byear{2012})
\doiurl{10.1016/j.ejor.2011.08.027}
\end{barticle}
\endbibitem

\bibitem[\protect\citeauthoryear{Camponogara et~al.}{2024}]{Camponogara:Flow-Split:2024}
\begin{barticle}
\bauthor{\bsnm{Camponogara}, \binits{E.}},
\bauthor{\bsnm{Seman}, \binits{L.O.}},
\bauthor{\bsnm{Muller}, \binits{E.R.}},
\bauthor{\bsnm{Gaspari}, \binits{L.K.} \bsuffix{Eduardo F.~Miyatake}},
\bauthor{\bsnm{Vieira}, \binits{B.F.}}:
\batitle{A {ReLU}-based linearization approach for maximizing oil production in subsea platforms: An application to flow splitting}.
\bjtitle{Chemical Engineering Science}
\bvolume{295},
\bfpage{120165}
(\byear{2024})
\doiurl{10.1016/j.ces.2024.120165}
\end{barticle}
\endbibitem

\bibitem[\protect\citeauthoryear{Vogel}{1968}]{Vogel:JPT:1968}
\begin{barticle}
\bauthor{\bsnm{Vogel}, \binits{J.V.}}:
\batitle{Inflow performance relationships for solution-gas drive wells}.
\bjtitle{Journal of Petroleum Technology}
\bvolume{20}(\bissue{01}),
\bfpage{83}--\blpage{92}
(\byear{1968})
\doiurl{10.2118/1476-PA}
\end{barticle}
\endbibitem

\bibitem[\protect\citeauthoryear{Kilin{\c{c}} and Sahinidis}{2018}]{baron2018}
\begin{barticle}
\bauthor{\bsnm{Kilin{\c{c}}}, \binits{M.R.}},
\bauthor{\bsnm{Sahinidis}, \binits{N.V.}}:
\batitle{Exploiting integrality in the global optimization of mixed-integer nonlinear programming problems with {BARON}}.
\bjtitle{Optimization Methods \& Software}
\bvolume{33}(\bissue{3}),
\bfpage{540}--\blpage{562}
(\byear{2018})
\doiurl{10.1080/10556788.2017.1350178}
\end{barticle}
\endbibitem

\bibitem[\protect\citeauthoryear{Bolusani et~al.}{2024}]{SCIP}
\begin{botherref}
\oauthor{\bsnm{Bolusani}, \binits{S.}},
\oauthor{\bsnm{Besan{\c{c}}on}, \binits{M.}},
\oauthor{\bsnm{Bestuzheva}, \binits{K.}},
\oauthor{\bsnm{Chmiela}, \binits{A.}},
\oauthor{\bsnm{Dion{\'{i}}sio}, \binits{J.}},
\oauthor{\bsnm{Donkiewicz}, \binits{T.}},
\oauthor{\bsnm{Doornmalen}, \binits{J.}},
\oauthor{\bsnm{Eifler}, \binits{L.}},
\oauthor{\bsnm{Ghannam}, \binits{M.}},
\oauthor{\bsnm{Gleixner}, \binits{A.}},
\oauthor{\bsnm{Graczyk}, \binits{C.}},
\oauthor{\bsnm{Halbig}, \binits{K.}},
\oauthor{\bsnm{Hedtke}, \binits{I.}},
\oauthor{\bsnm{Hoen}, \binits{A.}},
\oauthor{\bsnm{Hojny}, \binits{C.}},
\oauthor{\bsnm{Hulst}, \binits{R.}},
\oauthor{\bsnm{Kamp}, \binits{D.}},
\oauthor{\bsnm{Koch}, \binits{T.}},
\oauthor{\bsnm{Kofler}, \binits{K.}},
\oauthor{\bsnm{Lentz}, \binits{J.}},
\oauthor{\bsnm{Manns}, \binits{J.}},
\oauthor{\bsnm{Mexi}, \binits{G.}},
\oauthor{\bsnm{M\"{u}hmer}, \binits{E.}},
\oauthor{\bsnm{Pfetsch}, \binits{M.E.}},
\oauthor{\bsnm{Schl{\"o}sser}, \binits{F.}},
\oauthor{\bsnm{Serrano}, \binits{F.}},
\oauthor{\bsnm{Shinano}, \binits{Y.}},
\oauthor{\bsnm{Turner}, \binits{M.}},
\oauthor{\bsnm{Vigerske}, \binits{S.}},
\oauthor{\bsnm{Weninger}, \binits{D.}},
\oauthor{\bsnm{Xu}, \binits{L.}}:
{The SCIP Optimization Suite 9.0}.
Technical report,
Optimization Online
(February 2024).
\url{https://optimization-online.org/2024/02/the-scip-optimization-suite-9-0/}
\end{botherref}
\endbibitem

\bibitem[\protect\citeauthoryear{Bezanson et~al.}{2017}]{Julia-2017}
\begin{barticle}
\bauthor{\bsnm{Bezanson}, \binits{J.}},
\bauthor{\bsnm{Edelman}, \binits{A.}},
\bauthor{\bsnm{Karpinski}, \binits{S.}},
\bauthor{\bsnm{Shah}, \binits{V.B.}}:
\batitle{Julia: A fresh approach to numerical computing}.
\bjtitle{SIAM {R}eview}
\bvolume{59}(\bissue{1}),
\bfpage{65}--\blpage{98}
(\byear{2017})
\doiurl{10.1137/141000671}
\end{barticle}
\endbibitem

\bibitem[\protect\citeauthoryear{Lubin et~al.}{2023}]{Lubin2023}
\begin{barticle}
\bauthor{\bsnm{Lubin}, \binits{M.}},
\bauthor{\bsnm{Dowson}, \binits{O.}},
\bauthor{\bsnm{{Dias Garcia}}, \binits{J.}},
\bauthor{\bsnm{Huchette}, \binits{J.}},
\bauthor{\bsnm{Legat}, \binits{B.}},
\bauthor{\bsnm{Vielma}, \binits{J.P.}}:
\batitle{{JuMP} 1.0: {R}ecent improvements to a modeling language for mathematical optimization}.
\bjtitle{Mathematical Programming Computation}
(\byear{2023})
\doiurl{10.1007/s12532-023-00239-3}
\end{barticle}
\endbibitem

\bibitem[\protect\citeauthoryear{Dantzig}{1960}]{dantzigSignificanceSolvingLinear1960}
\begin{barticle}
\bauthor{\bsnm{Dantzig}, \binits{G.B.}}:
\batitle{On the significance of solving linear programming problems with some integer variables}.
\bjtitle{Econometrica}
\bvolume{28}(\bissue{1}),
\bfpage{30}--\blpage{44}
(\byear{1960})
\doiurl{10.2307/1905292}
\end{barticle}
\endbibitem

\bibitem[\protect\citeauthoryear{Beale}{1979}]{bealeBranchBoundMethods1979}
\begin{barticle}
\bauthor{\bsnm{Beale}, \binits{E.M.L.}}:
\batitle{Branch and bound methods for mathematical programming systems}.
\bjtitle{Annals of Discrete Mathematics}
\bvolume{5},
\bfpage{201}--\blpage{219}
(\byear{1979})
\doiurl{10.1016/S0167-5060(08)70351-0}
\end{barticle}
\endbibitem

\end{thebibliography}

\end{document}